\documentclass[final]{siamltex}

\pdfoutput=1

\usepackage{epsfig}
\usepackage[english]{babel}
\usepackage{amssymb,amsmath,amsfonts}
\usepackage{mathrsfs}
\usepackage{mdwlist}
\usepackage{graphicx,enumerate,url,hyperref,enumitem,multirow,multicol}
\usepackage[usenames,dvipsnames,svgnames]{xcolor}
\hypersetup{pdfauthor={Feng Xue},
    pdftitle={On the degenerate G-FBS operator},
    colorlinks=true,
    linkcolor=blue,
    citecolor=red,
    urlcolor=blue,
}

\usepackage{algorithm}
\usepackage{algcompatible}
\algnewcommand\algorithmicreturn{\textbf{return}}
\algnewcommand\RETURN{\State \algorithmicreturn}%
\usepackage[outercaption]{sidecap}

\usepackage{tikz,bm}

\newcommand{\tabincell}[2]{\begin{tabular}{@{}#1@{}}#2\end{tabular}}

\usepackage{pgfplots}
\pgfplotsset{
  tick label style={font=\footnotesize},
  label style={font=\footnotesize},
  legend style={font=\footnotesize}
}

\usepackage{booktabs}
\usepackage{caption,subcaption}

\usepackage{todonotes}
\newcommand{\xtilde}{\tilde{x}}
\newcommand{\xbar} {\bar{x}}
\newcommand{\cSbar} { {\bar{\mathcal{S}}} }

\newcommand{\bbstar}{x^\star}
\newcommand{\cstar}{x^\star}

\newcommand{\rd}{\mathrm{d}}

\newcommand{\bb}{x}

\newcommand{\bs}{\mathbf{s}}
\newcommand{\bbS}{\mathbb{S}}

\newcommand{\ran}{\mathrm{ran}}

\newcommand{\Fix}{\mathrm{Fix}}

\newcommand{\R}{\mathbb{R}}
\newcommand{\N}{\mathbb{N}}
\newcommand{\cI}{\mathcal{I}}
\newcommand{\cR}{\mathcal{R}}
\newcommand{\cS}{\mathcal{S}}

\newcommand{\cO}{\mathcal{O}}
\newcommand{\cT}{\mathcal{T}}
\newcommand{\cA}{\mathcal{A}}
\newcommand{\cU}{\mathcal{U}}
\newcommand{\cV}{\mathcal{V}}
\newcommand{\cZ}{\mathcal{Z}}
\newcommand{\cY}{\mathcal{Y}}

\newcommand{\cG}{\mathcal{G}}
\newcommand{\cP}{\mathcal{P}}
\newcommand{\cJ}{J}

\newcommand{\cH}{\mathcal{H}}
\newcommand{\cM}{\mathcal{M}}

\newcommand{\cX}{\mathcal{X}}

\newcommand{\cQ}{\mathcal{Q}}

\newcommand{\cK}{\mathcal{K}}
\newcommand{\cB}{\mathcal{B}}

\newcommand{\cQtilde}{\tilde{\mathcal{Q}}}

\newcommand{\utilde}{\tilde{u}}
\newcommand{\stilde}{\tilde{s}}
\newcommand{\ftilde}{\tilde{f}}
\newcommand{\gtilde}{\tilde{g}}

\newcommand{\bbtilde}{\tilde{x}}

\newcommand{\zer}{\mathrm{zer}}
\newcommand{\dist}{\mathrm{dist}}
\newcommand{\dom}{\mathrm{dom}}

\newcommand{\sri}{\mathrm{sri}}

\newcommand{\weak} {\rightharpoonup}

\newcommand{\red} {\color{red}}

\newcommand{\be}{\begin{equation}}
\newcommand{\ee}{\end{equation}}

\DeclareMathOperator{\prox}{prox}       % proximal operator      
\DeclareMathOperator{\Arg}{Arg}         % Argument
\newtheorem{assumption}{Assumption}
\newtheorem{fact}{Fact}
\newtheorem{example}{Example}
\newtheorem{remark}{Remark}
\usepackage[normalem]{ulem} % for sout, can remove for final version

\newif\ifcompilePGFfigs
\compilePGFfigstrue

\graphicspath{{figs/}}

\title{A generalized  forward-backward splitting operator:  Degenerate analysis and applications}

\author{Feng Xue\thanks{National Key Laboratory, Beijing, China (\tt{fxue@link.cuhk.edu.hk}).}  }
\date{\today}

\begin{document}

\maketitle
%\large

\begin{abstract}
In this paper, we study the  nonexpansive properties of  a generalized forward-backward splitting (G-FBS) operator, particularly under the setting of degenerate metric, from which follow the convergence results in terms of degenerate metric of the associated fixed-point iterations. The descent lemma and sufficient decrease property are also extended to degenerate case. It is further shown that  the G-FBS operator provides a simplifying and unifying framework to model and analyze a great variety of operator splitting algorithms,  many existing results are exactly recovered or relaxed from our general results. 
\end{abstract}

% REQUIRED

% REQUIRED

\section{Introduction}  \label{sec_intro}

\subsection{Forward-backward splitting algorithms}
A classical optimization problem is to find a zero point of the sum of a (set-valued) maximally monotone operator $\cA: \cH \mapsto 2^\cH$ and a $(1/\beta)$-cocoercive operator $\cB: \cH \mapsto \cH$ with $\beta \in [0,+\infty[$\footnote{In our context, the problem \eqref{inclusion}  encompasses the case of $\beta=0$, which corresponds to $\cB=0$, that can be understood as $+\infty$-cocoercive. See further Remark \ref{r_assume_1}-(iv).}:
\be \label{inclusion}
 0 \in (\cA +\cB) \bbstar.
\ee
{\red A standard solver of the problem \eqref{inclusion} is the classical forward-backward splitting algorithm \cite{attouch_2018}:
\be \label{fbs}
\bb^{k+1} := \big(\cA + \frac{1}{\tau} \cI \big)^{-1} 
\big(\frac{1}{\tau} \cI -  \cB \big)  \bb^k .
\ee
We in this paper study a generalized version of \eqref{fbs}, by extending $\frac{1}{\tau} \cI$ to more general metric $\cQ$
 \cite{plc_vu_2014}:
\be \label{gfbs}
\bb^{k+1} := \cT  \bb^k, \quad \text{where\ }
\cT =  (\cA +  \cQ)^{-1} (\cQ -  \cB), 
\ee
where $\cT$ is called a generalized forward-backward splitting (G-FBS) operator.}

The scheme \eqref{gfbs} unifies and extends several typical iterative methods shown in Table 1.1.   The proximal FBS  has been extensively studied in the literature, e.g., \cite{plc,plc_chapter}, in a convex setting, and further discussed in \cite{bolte_2009,bolte_2013,bolte_2014} for the nonconvex case. The metric proximal FBS was recently studied in  \cite{pesquet_2016,pesquet_2014} in a nonconvex setting.

\begin{table} [h!] \label{table}
\centering
\caption{The related schemes to G-FBS \eqref{gfbs}  }
\vspace{-.5em}
\hspace*{-.2cm}
\resizebox{.97\columnwidth}{!} {
\begin{tabular}{|l|l|l|} 
    \hline   
algorithms & iterative schemes & correspondence with \eqref{gfbs}  \\
\hline
 \tabincell{l}{ classical  \\ FBS  \cite{attouch_2018,plc_vu_2014} }   &
 \tabincell{l} { $\bb^{k+1} := (\cI +\tau \cA)^{-1} (\cI -\tau \cB) \bb^k $  \\ \hspace*{.7cm} $:= \cJ_{\tau \cA} \circ (\cI - \tau \cB) \bb^k $ }   & $\cQ = \frac{1}{\tau} \cI$ \\ 
    \hline  
 \tabincell{l}{ proximal  \\ FBS   \cite{plc,plc_chapter} }   &
 \tabincell{l} { $
 \bb^{k+1} := (\cI +\tau \partial g)^{-1} (\cI - \tau \nabla f) \bb^k $  \\ \hspace*{.7cm} $ :=  \prox_{\tau g} \big( \bb^k - \tau \nabla f(\bb^k) \big) $ }   & 
  $\cA = \partial g$, $\cB = \nabla f$, $\cQ=\frac{1}{\tau}\cI$  \\ 
    \hline  
 \tabincell{l}{   metric proximal   \\ FBS    \cite{pesquet_2016,pesquet_2014,repetti} }   &
 \tabincell{l} { $ \bb^{k+1} : = (\partial g +   \cQ)^{-1} (\cQ -  \nabla f ) \bb^k  $  \\ \hspace*{.7cm} $ := \prox_g^\cQ \big( \bb^k - \cQ^{-1} \nabla  f(\bb^k) \big) $ }   & 
  $\cA = \partial g$, $\cB = \nabla f$ \\ 
    \hline  
 \tabincell{l}{    PPA \cite{rtr_1976,ppa_guler}   }   &
 \tabincell{l} { $ x^{k+1} : = (\cI + \tau \partial g )^{-1} x^k  $  \\ \hspace*{.7cm} $ :=\prox_{\tau g} (x^k) $ }   & 
  $\cA = \partial g$, $\cB = 0$, $\cQ = \frac{1}{\tau} \cI$ \\ 
    \hline  
 \tabincell{l}{  metric   PPA  \cite{warp}   }   &
 \tabincell{l} { $ x^{k+1} : = (\cA +   \cQ)^{-1}  \cQ   x^k  $ }   &     $\cB = 0$  \\ 
    \hline  
     \end{tabular}  } 
\vskip 0.5em
\end{table}

The first-order operator splitting algorithms have been revitalized in the past decade, due to the high demand in solving large and huge scale problems arising in a wide range of fundamental applications (e.g., signal processing and machine learning).
Several frameworks and tools have recently been proposed for the unification of these algorithms, e.g., nonexpansive operator  \cite{ljw_mapr}, Fej\'{e}r monotonicity \cite{plc_vu} and  fixed-point theory \cite{plc_fixed}. However, these concepts are rather abstract, and not directly connected to the specific algorithms at hand. It is not easy to reinterpret many complicated splitting algorithms using these tools. The works of \cite{teboulle_2018,plc_bregman} discussed the Bregman proximal mapping and the associated Bregman proximal gradient algorithm. It remains unclear  how to use the framework to analyze the splitting  algorithms.

Asymmetric forward-backward adjoint (AFBA) splitting scheme was shown in \cite{latafat_2017,latafat_chapter} to be a unified structure for many splitting algorithms. However, it fails to provide a standard pattern of convergence analysis for these algorithms.  In a series of works of He and Yuan \cite{hbs_siam_2012,hbs_siam_2012_2,hbs_jmiv_2017}, the metric PPA have been shown as a unified framework for Douglas-Rachford splitting (DRS), alternating direction methods of multipliers (ADMM) and primal-dual hybrid gradient (PDHG) algorithms, by the characterization of variational inequality. This approach was further extended in \cite{fxue_gopt}, which unifies and simplifies the analysis of these algorithms based on an equivalent inclusion form. The problem setting, which was characterized by the variational inequality in  \cite{hbs_siam_2012,hbs_siam_2012_2,hbs_jmiv_2017} on a case-by-case basis, is  completely and uniformly encoded in the maximally monotone operator $\cA$ in \cite{fxue_gopt}. The similar idea of using monotone operator theory was also developed in a most recent book draft \cite{ywt_book}.  However, this method fails to cover the proximal FBS \cite{plc} and many primal-dual splitting (PDS) algorithms \cite{plc_vu_2014,condat_2013,plc_2012},  which involve a Lipschitz continuous gradient. This limitation was also mentioned in \cite[Sect. 8]{fxue_gopt}.

The purpose of this paper is to show that the G-FBS operator \eqref{gfbs} provides a unified framework and analysis for various splitting algorithms.  However, we found that  the metric $\cQ$ associated with several splitting algorithms is often positive {\it semi-}definite, rather than {\it strictly} positive definite (see Examples \ref{eg_radmm}, \ref{eg_alm} and \ref{eg_lb} in Sect. \ref{sec_eg}). This phenomenon is referred to as ‘{\it degeneracy}’, which implies that  the information of $x$ lying in the non-trivial kernel space of $\cQ$ is redundant and inactive  in the scheme \eqref{gfbs}. Consequently,  the convergence of $\{x^k\}_{k\in\N}$ generated by \eqref{gfbs} in the whole space may not hold, since the distance in a sense of the degenerate metric\footnote{Correspondingly, the term {\it non-degenerate} used in this paper exactly means {\it strictly positive definite}.} cannot measure the closeness between two points in whole space. Thus, it is not a trivial extension from the classical scalar case of $\cQ=\frac{1}{\tau}\cI$ and deserves particular attention.  Many recent works related to variable metric version of algorithms assumed the metric $\cQ$ to be non-degenerate, e.g., \cite{pesquet_2016,pesquet_2014,repetti,
fxue_gopt,condat_tour}. The degenerate case, to the best of our knowledge, has never been discussed before. This is the focus of this paper.

\subsection{Contributions} 
The contributions are in order.
\begin{itemize}
\item We study the nonexpansive properties of the G-FBS operator  \eqref{gfbs} and its relaxed version \eqref{rgfbs} under degenerate setting (cf. Sect. \ref{sec_operator}). The weak convergence and metric-based asymptotic regularity of the associated fixed-point iterations are established in Sect. \ref{sec_gfbs}. All of the existing results presented in \cite{plc_vu_2014,pesquet_2016,pesquet_2014,repetti} are extended to the  non-degenerate case.

\item We generalize the descent lemma and sufficient decrease property to the degenerate setting (cf. Lemmas \ref{l_descent} and \ref{l_decrease}), and then, the convergence in terms of objective value is established in Proposition \ref{p_gpfbs_obj}. Many related results under non-degenerate setting can be exactly recovered from our results.

\item In Sect. \ref{sec_extension}, the G-FBS operator is further generalized to arbitrary relaxation operator $\cM$, the convergence behaviours of the relaxed G-FBS algorithm \eqref{gppa} are discussed. 

\item It is shown in Sect. \ref{sec_eg} that a great variety of popular algorithms can be uniformly represented by the G-FBS operator \eqref{gfbs} or its relaxed version \eqref{t_relaxed}, by specifying the  operators $\cA$ and $\cB$, (degenerate) metric $\cQ$ (and relaxation operator  $\cM$, if necessary). There are three big advantages over other existing frameworks: (1) it is easy to fit specific algorithm into the G-FBS operator \eqref{gfbs}, without basically changing the algorithmic structures and exploring the contractive properties as in \cite{ljw_mapr,plc_vu}; (2) the convergence analysis is unified in Sect. \ref{sec_gfbs} and \ref{sec_extension}: one does not need to perform the {\it ad-hoc} analysis for specific algorithms as in \cite{plc_fixed,condat_tour}; (3) it covers proximal FBS and its related algorithms, which cannot be analyzed by metric PPA framework of \cite{fxue_gopt,ywt_book}.  This unification and simplification provides a much easier way to  understand these algorithms, compared to the original proofs in  literature.
\end{itemize}

\subsection{Notations and definitions} \label{sec_notation}
We use standard notations and concepts from convex analysis and variational analysis, which, unless otherwise specified, can all be found in the classical and recent monographs \cite{rtr_book,rtr_book_2,plc_book,beck_book}.

A few more words about our  notations are in order. Let $\cH$ be a real Hilbert space, equipped with inner product $\langle \cdot |\cdot \rangle$ and induced norm $\|\cdot \|$. The classes of linear, self-adjoint, self-adjoint and positive (semi-)definite  operators are denoted by $\bbS$, $\bbS_+$ and $\bbS_{++}$, respectively. The $\cQ$-norm is defined as: $\|\cdot\|_\cQ^2 :=  \langle \cQ\cdot | \cdot \rangle$. Note that  $\cQ$ here is allowed to be degenerate, and thus, $\|x\|_\cQ=0$ does not necessarily imply $x=0$.  $\cQ^\top$ denotes the adjoint of $\cQ$.  $\cQ^\dagger$ stands for the pseudo-inverse of $\cQ$, if $\cQ$ is singular. The strong and weak convergences are denoted by $\rightarrow$ and $\weak$, respectively.  {\red Following \cite[Definition 11.3]{plc_book}, the set of minimizers of a function $f$ is denoted by $\Arg\min f$. If $\Arg \min f$ is a singleton, its unique element is denoted by $\arg\min_x f(x)$. }

\vskip.1cm
Note that our expositions will be largely based on the nonexpansive  properties in the context of   arbitrary degenerate  metric $\cQ$. Thus, it is necessary to extend  the classical notions of  Lipschitz continuity \cite[Definition 1.47]{plc_book}, nonexpansiveness \cite[Definition 4.1]{plc_book}, cocoerciveness \cite[Definition 4.10]{plc_book} and averagedness \cite[Definition 4.33]{plc_book} 
 to this setting.

\begin{definition} \label{def_nonexpansive}
Let the metric $\cQ$ be self-adjoint and at least degenerate. Then, the operator $\cT$ is:
\begin{itemize}
\item[\rm (i)]  {\it $\cQ$-based $\xi$-Lipschitz continuous}, if
$ \big\| \cT \bb_1 - \cT \bb_2  \big\|_\cQ \le \xi
 \big\|  \bb_1 - \bb_2 \big\|_\cQ$.

\item[\rm (ii)]  {\it $\cQ$-nonexpansive}, if
$ \big\| \cT \bb_1 - \cT \bb_2  \big\|_\cQ \le 
 \big\|  \bb_1 - \bb_2 \big\|_\cQ$.

\item[\rm (iii)] {\it $\cQ$-firmly nonexpansive}, if 
\[
\langle \cQ (\bb_1 - \bb_2) | \cT\bb_1-\cT\bb_2 \rangle
 \ge  \big\| \cT \bb_1 - \cT \bb_2  \big\|_\cQ^2,
\]
or equivalently, 
\[
\big\| \cT\bb_1-\cT\bb_2 \big\|_\cQ^2 +  
\big\|(\cI - \cT) \bb_1 - (\cI - \cT) \bb_2  \big\|_\cQ^2 \le 
 \big\| \bb_1 - \bb_2  \big\|_\cQ^2.
 \]

\item[\rm (iv)]  {\it $\cQ$-based $\beta$-cocoercive}, if $\beta \cT$ is $\cQ$-firmly nonexpansive:
\[
\langle \cQ (\bb_1 - \bb_2) | \cT\bb_1-\cT\bb_2 \rangle
\ge \beta \big\| \cT \bb_1 - \cT \bb_2  \big\|_\cQ^2.
 \]

\item[\rm (v)]  {\it $\cQ$-based $\alpha$-averaged} with $\alpha \in \ ]0,1 [$, if there exists a $\cQ$-nonexpansive operator $\cK$, such that $\cT = (1-\alpha) \cI + \alpha \cK$.
\end{itemize}
\end{definition} 

\vskip.1cm
It is easy to check that the standard results of \cite[Remark 4.34, Propositions 4.35, 4.39, 4.40]{plc_book} also hold for the (degenerate) metric $\cQ$. This is merely a trivial extension, and the degeneracy of $\cQ$ is not troublesome here.  For example, similar to \cite[Proposition 4.35]{plc_book},  the $\cQ$-based $\alpha$-averaged $\cT$ should satisfy
\be \label{ave}
\big\|\cT x_1 -\cT x_2 \big\|_\cQ^2 \le 
\big\| x_1 -  x_2 \big\|_\cQ^2
- \frac{1-\alpha}{\alpha} \big\| (\cI-\cT) x_1 
-(\cI-\cT) x_2 \big\|_\cQ^2.
\ee

\section{The nonexpansive properties under  degenerate setting}
\label{sec_operator}
\subsection{Assumptions and basic results}
In this part, we make  the following  assumption regarding the operator $\cT$ in \eqref{gfbs}, with particular focus on the degenerate metric.
\begin{assumption} [{\red Degenerate} setting]  \label{assume_1}
\begin{itemize}
\item [\rm (i)] $\cA$ is (set-valued) maximally monotone;
\item [\rm (ii)] $\cQ$ is linear, self-adjoint and bounded;
\item [\rm (iii)] $\cQ$ is  degenerate, such that $\ker \cQ \backslash \{0\} \ne \emptyset$;
\item [\rm (iv)] $\cQ$ has a closed range, i.e., $\overline{\ran \cQ} = \ran \cQ$;
 \item [\rm (v)] $\cB$ is $\beta^{-1}$-cocoercive with $\beta \in\ [0,+\infty[$, and $\dom \cB =\cH$;
 \item [\rm (vi)]   $\ran \cB \subseteq \ran \cQ$;
\item [\rm (vii)] $\ran(\cA+\cQ) \supseteq \ran (\cQ-\cB)$;
\item [\rm (viii)] $\zer(\cA+\cB) \ne \emptyset$.
\end{itemize}
\end{assumption}

One will see in Sect. \ref{sec_eg} that Assumption \ref{assume_1} is rather mild, which can be  satisfied by most typical splitting algorithms. We further have the following remarks regarding Assumption \ref{assume_1}.

\begin{remark} \label{r_assume_1}
{\rm (i)} Assumption \ref{assume_1}-(iv)  implies, by \cite[Fact 2.26]{plc_book}, that $\exists \nu \in\ ]0,+\infty[$,  $\| \cQ x\| \ge \nu \|x\|$, $\forall x \in (\ker \cQ)^\perp = \ran \cQ$. This shows: (1) $\cQ$ is strictly positive definite in $\ran \cQ$; (2) This constant $\nu$ can be understood as the smallest eigenvalue restricted to $\ran\cQ$, which will be frequently used in the remainder of this paper. If $\cQ$ is non-degenerate and closed, this assumption simply becomes: $\exists \nu \in\  ]0,+\infty[$, such that $\cQ \succeq \nu \cI$.

{\rm (ii)} Regarding  Assumption \ref{assume_1}-(v),  by the classical Baillon-Haddad Theorem \cite[Corollary 18.17]{plc_book}, $\cB$ is $\beta$-Lipschitz continuous.   By \cite[Example 20.31]{plc_book}, $\cB$ is also maximally monotone, and so is $\cA +\cB$ due to $\dom \cB=\cH$, by \cite[Corollary 25.5]{plc_book}. This observation will be essential for proving weak convergence in Theorem \ref{t_dist}.

{\rm (iii)} Assumption \ref{assume_1}-(vi) is essential for our degenerate analysis.  To understand this, one can think of a simple, but non-trivial example of $ \cB = \begin{bmatrix}
1 & 0 \\ 0 & 0 \end{bmatrix}$, which is 1-cocoercive, and satisfies $\ran \cB \subseteq \ran \cQ$ for a degenerate metric $ \cQ = \begin{bmatrix}
1 & 0 \\ 0 & 0 \end{bmatrix}$.

{\rm (iv)}  Assumption \ref{assume_1}-(v) and (vi) also encompass a trivial, but important case of $\cB=0$,   since $\cB=0$ is obviously $0$-Lipschitz continuous (i.e., $\beta=0$), and thus $+\infty$-cocoercive. Also note  $\ran \cB =\{0\} \subseteq \ran \cQ$ for any linear operator $\cQ$. This special case reduces the G-FBS operator \eqref{gfbs}   to $(\cA+\cQ)^{-1}\cQ$---the warped resolvent \cite{warp} or $D$-resolvent \cite{hhb_resolvent,hhb_2003}, which has many important applications in DRS, ADMM and PDHG  (see \cite{fxue_gopt} and also Examples \ref{eg_ppa}, \ref{eg_cp}, \ref{eg_plc_dual}, \ref{eg_alm}, \ref{eg_linear_alm} and \ref{eg_lb} in Sect. \ref{sec_eg}).

{\rm (v)} Assumption \ref{assume_1}-(vii) guarantees that $\cT$ given by \eqref{gfbs} is well-defined everywhere, i.e., $\dom \cT = \cH$. This assumption was carefully discussed in \cite{warp,arias_infimal}.   If $\cQ$ is non-degenerate, by the classical Minty's theorem \cite[Theorem 21.1]{plc_book}, one can conclude $\ran(\cA+\cQ) = \cH$ for any maximally monotone operator $\cA$. This automatically fulfills this assumption. However, this assumption is needed here, since  $\cA+\cQ$ may not be surjective in $\cH$ for degenerate $\cQ$.

{\rm (vi)} Assumption \ref{assume_1}-(viii)  ensures the existence of the solution to \eqref{inclusion}. 
\end{remark}

The following basic results are essential for later developments.
\begin{fact} \label{f_1}
Under Assumption \ref{assume_1},  the following hold. 
\begin{itemize}
\item[\rm (i)] $ \big\| \cB  x   \big\|^2  \ge  \nu  \big\|  \cQ^\dagger \cB x  \big\|_{\cQ}^2$, $\forall x\in\cH$;
\item[\rm (ii)] 
$\big\langle  \cB  x | y    \big\rangle  
  \le  \big\|  \cQ^\dagger \cB x  \big\|_{\cQ} \cdot 
\big\| y\big\|_\cQ$, $\forall (x,y) \in \cH \times \cH$;
\item[\rm (iii)] 
$\big\langle  \cB  x | y    \big\rangle  
  \le \eta  \big\|  \cQ^\dagger \cB x  \big\|_{\cQ}^2
  +\frac{1}{4\eta}  \big\| y\big\|_\cQ^2$, $\forall \eta\in\ ]0,+\infty[$, $\forall (x,y) \in \cH \times \cH$, 
\end{itemize}
where $\nu$ is defined in Remark \ref{r_assume_1}-(i).
\end{fact}
\begin{proof}
(i) Since $\cB x\in \ran \cQ$, we have
\begin{eqnarray}
 \big\| \cB  x   \big\|^2 &=& 
 \big\|\cP_{\ran\cQ} ( \cB  x )\big\|^2 
\quad \text{[$\cP_{\ran\cQ}$---projection onto $\ran\cQ$]}
\nonumber \\
&= &  \big\| \cQ \cQ^\dagger   \cB  x \big\|^2 
   \quad \text{[by $\cP_{\ran \cQ} = \cQ \cQ^\dagger$]}
\nonumber\\
& \ge & \nu  \big\|  \cQ^\dagger \cB x  \big\|_{\cQ}^2. 
\quad \text{[by Remark \ref{r_assume_1}-(i)]}
\nonumber
\end{eqnarray}

(ii) $\big\langle  \cB  x | y    \big\rangle  = 
 \big\langle   \cQ \cQ^\dagger  \cB  x | y    \big\rangle 
 =    \big\langle  \sqrt{ \cQ} \cQ^\dagger  \cB  x | 
 \sqrt{ \cQ} y  \big\rangle 
  \le  \big\|  \cQ^\dagger \cB x  \big\|_{\cQ} \cdot 
\big\| y\big\|_\cQ$.

(iii) By (ii) and {\red Fenchel--Young inequality}\footnote{\red The Fenchel--Young inequality reads $f(a)+f^*(b) \ge \langle a|b\rangle$ for a proper function $f$ \cite[Proposition 13.15]{plc_book}. Particularly, if $f = \eta \|\cdot\|^2$ with $\eta>0$, it becomes $\langle a|b \rangle \le \eta\|a\|^2 + \frac{1}{4\eta} \|b\|^2$, $\forall \eta >0$, from which follows Fact \ref{f_1}-(iii).}, we have,  $\forall \eta\in\ ]0,+\infty[$:
\[
\big\langle  \cB  x | y    \big\rangle  = 
    \big\langle  \sqrt{ \cQ} \cQ^\dagger  \cB  x | 
 \sqrt{ \cQ} y  \big\rangle 
  \le \eta  \big\|  \cQ^\dagger \cB x  \big\|_{\cQ}^2
  +\frac{1}{4\eta}  \big\| y\big\|_\cQ^2.
\]
\hfill 
\end{proof}

\begin{lemma} \label{l_t}
Under Assumption \ref{assume_1}, the G-FBS operator given in \eqref{gfbs} can also be written as
\[
\cT = (\cA+\cQ)^{-1} \cQ \circ (\cI - \cQ^\dagger \cB).
\]
\end{lemma}
\begin{proof}
We have, $\forall  x \in  \cH$:
\begin{eqnarray}
\cT x &=& (\cA+\cQ)^{-1} (\cQ-\cB) x 
= (\cA+\cQ)^{-1} (\cQ x - \cB x) 
\nonumber \\
&= & (\cA+\cQ)^{-1} (\cQ x - \cP_{\ran \cQ} (\cB x) ) 
\quad \text{[by $\ran \cB \subseteq \ran \cQ$]}
\nonumber \\
&= & (\cA+\cQ)^{-1} (\cQ x - \cQ  \cQ^\dagger \cB x ) 
\quad \text{[by $\cP_{\ran \cQ} = \cQ \cQ^\dagger$]}
\nonumber \\
&= & (\cA+\cQ)^{-1} \cQ ( x -  \cQ^\dagger \cB x ).
\nonumber
\end{eqnarray}
Thus, the desired result follows, since $x$ is arbitrary. \hfill
\end{proof}

\vskip.1cm
In particular, if $\cQ$ is non-degenerate, it is easy to derive 
$ \cT = (\cI+\cQ^{-1}\cA)^{-1}  (\cI - \cQ^{-1} \cB)
= J_{\cQ^{-1}\cA} \circ  (\cI - \cQ^{-1} \cB)$, where $J_{\cQ^{-1}\cA} $ denotes the resolvent of $\cQ^{-1}\cA$. However, for degenerate case, $(\cA+\cQ)^{-1}\cQ$ cannot be viewed as a well-defined resolvent since $\ran(\cA+\cQ) \supseteq \ran\cQ$ and  $\ran(\cA+\cQ) =\cH$ may not hold.

\vskip.1cm
\begin{lemma} \label{l_cocoercive}
Under Assumption \ref{assume_1},  the following hold.
\begin{itemize}
\item[\rm (i)] $\cQ^{\dagger} \cB$ is $\cQ$-based  $ \frac{\nu} {\beta}$-cocoercive; 
\item[\rm (ii)] $\cQ^{\dagger} \cB$ is $\cQ$-based  $ \frac{\beta} {\nu} $-Lipschitz continuous,
\end{itemize}
where $\nu$ is specified in Remark \ref{r_assume_1}-(i).
\end{lemma}
\begin{proof}
We deduce that:
\begin{eqnarray}
 \big\| \bb_1 - \bb_2 \big\|_\cQ \cdot  
\big\| \cQ^\dagger \cB \bb_1 - \cQ^\dagger \cB \bb_2 \big\|_\cQ 
& \ge &   \big\langle \bb_1 - \bb_2 \big| 
\cB \bb_1 - \cB \bb_2 \big\rangle 
\quad \text{[by Fact \ref{f_1}-(ii)]}
\nonumber \\ 
& \ge &  \frac{1}{\beta} \big\| \cB \bb_1 - \cB \bb_2 \big\|^2
\quad \text{[by cocoerciveness of $\cB$]}
\nonumber\\
& \ge &  \frac{\nu}{\beta} \big\|  \cQ^\dagger \cB \bb_1 - 
\cQ^\dagger \cB \bb_2 \big\|_{\cQ}^2. 
\quad \text{[by Fact \ref{f_1}-(i)]}
\nonumber 
\end{eqnarray}
The proof is completed by Definition \ref{def_nonexpansive}.
\hfill
\end{proof}

\vskip.1cm
Lemma \ref{l_cocoercive} extends the classical Baillon-Haddad theorem \cite[Corollary 18.17]{plc_book} to the degenerate metric case. If $\cQ$ is non-degenerate, this result still holds, by replacing the pseudo-inverse $\cQ^\dagger$ with a simple inverse $\cQ^{-1}$. In addition, Lemma \ref{l_cocoercive} gives an answer to the question raised in \cite[Theorem 3.2]{condat_tour}, where it is required to check the cocoerciveness of $\cQ^{-1}\cB$ in the context of  non-degenerate $\cQ$.

\subsection{Nonexpansive properties}
The nonexpansive properties are given below.
\begin{lemma} \label{l_T}
Given the operator $\cT: \cH\mapsto \cH$ defined in \eqref{gfbs}, then, under Assumption \ref{assume_1}, the following hold.
\begin{itemize}
\item [\rm (i)]    $\big\| \cT \bb_1 - \cT \bb_2  \big\|_\cQ^2 \le 
 \langle \cQ( \bb_1 - \bb_2) | \cT\bb_1 - \cT \bb_2 \rangle
-  \langle \cB \bb_1 - \cB \bb_2 | \cT\bb_1-\cT\bb_2 \rangle$;

\item[\rm (ii)] $\langle \cQ( \bb_1 - \bb_2) | \cT\bb_1 - \cT \bb_2 \rangle  \ge   \big\| \cT \bb_1 - \cT \bb_2  \big\|_\cQ^2   - \frac{\beta}{4\nu} \big\| (\cI - \cT)  \bb_1 -
 (\cI - \cT) \bb_2  \big\|_\cQ^2$; 

\item[\rm (iii)] $ \big\| \cT \bb_1 - \cT \bb_2  \big\|_\cQ^2  
+ \Big( 1-\frac{\beta}{2 \nu }  \Big)
 \big\|(\cI- \cT) \bb_1 - (\cI -\cT) \bb_2  \big\|_\cQ^2  
\le  \big\| \bb_1 -  \bb_2  \big\|_\cQ^2 $,
\end{itemize}
where $\nu$ is defined in Remark \ref{r_assume_1}-(i).
\end{lemma}

\begin{proof} 
(i) We develop
\begin{eqnarray}
&&  \big\| \cT \bb_1 - \cT \bb_2  \big\|_\cQ^2 
\nonumber \\  &=& \big  \langle \cQ \cT \bb_1 - \cQ \cT \bb_2 | \cT \bb_1 - \cT \bb_2 \big  \rangle 
\nonumber \\ 
&  \le &\big \langle \cQ \cT \bb_1  - \cQ \cT \bb_2 | \cT \bb_1- \cT \bb_2 \big  \rangle + 
\big \langle \cA \cT \bb_1 - \cA \cT \bb_2 | \cT\bb_1 - \cT\bb_2 \big \rangle 
 \quad \text{[by monotonicity of $\cA$]}
\nonumber \\ &= & 
\big  \langle (\cQ  - \cB) \bb_1 - (\cQ - \cB) \bb_2 | \cT\bb_1 - \cT\bb_2\big  \rangle
\quad \text{[since $\cQ - \cB \in \cQ \cT +\cA \cT $ by \eqref{gfbs}]}
\nonumber \\ 
& = & 
 \big \langle \cQ(\bb_1 - \bb_2) | \cT\bb_1-\cT\bb_2 \big  \rangle
- \big \langle \cB \bb_1 - \cB \bb_2 | \cT\bb_1-\cT\bb_2 \big \rangle.
 \nonumber 
\end{eqnarray}

\vskip.1cm
(ii) Denoting $\cR :=\cI - \cT$, we deduce that
\begin{eqnarray}
&&  \big \langle \cB \bb_1 - \cB \bb_2 | \cT\bb_1 - \cT \bb_2  \big \rangle  
\nonumber \\  &=& \big \langle \cB \bb_1 - \cB \bb_2 | \bb_1 -   \bb_2  \big \rangle -
 \big \langle \cB \bb_1 - \cB \bb_2 | \cR\bb_1 - \cR \bb_2  \big \rangle
  \quad \text{[by  $\cT = \cI - \cR$]}
\nonumber \\ 
&  \ge & \frac{\nu} {\beta} \big\|\cQ^\dagger \cB \bb_1  - \cQ^\dagger \cB \bb_2 \big\|_\cQ^2 
-  \big \langle \cB \bb_1 - \cB \bb_2 | \cR\bb_1 - \cR \bb_2  \big \rangle \quad \text{[by Lemma \ref{l_cocoercive}]}
\nonumber \\ 
&  \ge &  \frac{\nu} {\beta} \big\|\cQ^\dagger \cB \bb_1  - \cQ^\dagger \cB \bb_2 \big\|_\cQ^2  
- \frac{\nu}{\beta} \big\|\cQ^\dagger \cB \bb_1  -\cQ^\dagger \cB \bb_2 \big\|_\cQ^2 
- \frac{\beta}{4\nu} \big\| \cR \bb_1  - \cR \bb_2 \big\|_\cQ^2 
 \quad \text{[by Fact \ref{f_1}-(iii)]}
\nonumber \\ 
& = & - \frac{\beta}{4\nu} \big\| \cR \bb_1  - \cR \bb_2 \big\|_\cQ^2 .
\nonumber 
\end{eqnarray}
Then, (ii) follows by substituting the above into (i).

\vskip.1cm
(iii) We develop
\begin{eqnarray}
&& \big\| \cT \bb_1 - \cT \bb_2  \big\|_\cQ^2  
+ \big\|(\cI- \cT) \bb_1 - (\cI -\cT) \bb_2  \big\|_\cQ^2  
\nonumber \\ 
& = & 2  \big\| \cT \bb_1 - \cT \bb_2  \big\|_\cQ^2  
+ \big\| \bb_1 -  \bb_2  \big\|_\cQ^2  
-2 \big\langle \cQ(\bb_1 -  \bb_2) | \cT \bb_1 - \cT \bb_2 \big\rangle
\nonumber \\
& \le & \big\| \bb_1 -  \bb_2  \big\|_\cQ^2  
+\frac{\beta}{2 \nu}   \big\| (\cI -  \cT)  \bb_1 - 
(\cI -  \cT)  \bb_2 \big\|_\cQ^2,
\quad \text{[by (ii)]}
\nonumber 
\end{eqnarray}
which yields (iii). \hfill 
\end{proof}

\begin{theorem}  \label{t_T}
Let $\cT$ be defined as \eqref{gfbs}. Under Assumption \ref{assume_1}, if $\nu$ specified in Remark \ref{r_assume_1}-(ii) satisfies $\nu> \beta/  2 $,  then, the following hold.
\begin{itemize}
\item[\rm (i)] $\cT$ is $\cQ$-based $\frac{2\nu } {4 \nu - \beta  } $-averaged.

\item[\rm (ii)] $\cT$  is always $\cQ$-nonexpansive, and in particular,  $\cQ$-firmly nonexpansive, if and only if $\beta=0$.

\item[\rm (iii)]  $\cI - \cT$ is $\cQ$-based $(1-\frac{\beta}{4\nu})$-cocoercive.

\item[\rm (iv)] $\cI - \gamma(\cI - \cT)$ is $\cQ$-based $\frac{2 \gamma \nu }  {4 \nu - \beta} $-averaged, if  $\gamma \in \ ]0, 2-\frac{\beta}{2 \nu}[$. 

\item[\rm (v)] $\cI - \gamma(\cI - \cT)$  is $\cQ$-nonexpansive, and in particular, $\cQ$-firmly nonexpansive,  if    $\gamma \in\  ]0, 1-\frac{\beta}{4 \nu } [$.
\end{itemize}
\end{theorem}

{\red 
\begin{proof}
(i) Comparing Lemma \ref{l_T}--(iii) with \eqref{ave}, the averagedness $\alpha$ satisfies $\frac{1-\alpha}{\alpha} = 1-\frac{\beta}{2\nu}$, which yields $\alpha = \frac{2\nu } {4 \nu - \beta  } $.

\vskip.1cm
(ii) If $\nu>\beta/2$, the averagedness $\alpha = \frac{2\nu } {4 \nu - \beta} \in [\frac{1}{2}, 1[$, which shows that $\cT$ is always $\cQ$-nonexpansive, and in particular, $\cQ$-firmly nonexpansive, only when $\alpha=\frac{1}{2}$, i.e., $\beta = 0$.

\vskip.1cm
(iii) By \cite[Proposition 4.39]{plc_book}, $\cT$ is $\cQ$-based $\alpha$-averaged, if and only if $\cI - \cT$ is $\cQ$-based $\frac{1}{2\alpha}$-cocoercive. Substituting $\alpha = \frac{2\nu } {4 \nu - \beta  } $ according to (i) completes the proof.

\vskip.1cm
(iv)--(v): direct results of applying \cite[Proposition 4.40]{plc_book} to (i).
\hfill 
\end{proof} }

\begin{remark} \label{r_nonexpansive}
{\rm (i)} Lemma \ref{l_T} and Theorem \ref{t_T} also hold for non-degenerate case, where  the condition $\nu >\beta/ 2$  is simply equivalent to $\cQ \succ \frac{\beta}{2} \cI$.

{\rm (ii)} If $\beta=0$ (i.e., $\cB=0$), the basic nonexpansive properties of the warped resolvent $\cT = (\cA+\cQ)^{-1} \cQ$ are recovered: (1) $\cT$ is $\cQ$-based $\frac{1}{2}$-averaged ($\cQ$-firmly nonexpansive); (2)  $\cI - \gamma(\cI - \cT)$ is $\frac{\gamma}{2}$-averaged, if $\gamma \in\ ]0, 2[$.
\end{remark}

\subsection{Another view on the averagedness}
Lemma \ref{l_T} can also be verified by the composition of  $\cT$.

\begin{proposition} \label{p_com}
Given $\cT$ given as Lemma \ref{l_t}, under Assumption \ref{assume_1}, the following hold.
\begin{itemize}
\item[\rm (i)] $(\cA+\cQ)^{-1} \cQ$ is $\frac{1}{2}$-averaged (or simply $\cQ$-firmly nonexpansive).

\item[\rm (ii)] $\cI - \cQ^\dagger  \cB$ is $\cQ$-based $\frac {\beta} {2\nu }$-averaged, if $\nu > \beta/ 2$. In particular, if $\nu \ge  \beta$,  $\cI - \cQ^\dagger  \cB$ is $\cQ$-firmly nonexpansive.

\item[\rm (iii)]  $\cT$ is $\cQ$-based $\frac{2\nu} {4 \nu -\beta }$-averaged, if $\nu > \beta/ 2$. 
\end{itemize}
\end{proposition}

\begin{proof}
{\rm (i)} Remark \ref{r_nonexpansive}-(ii).

\vskip.1cm
(ii) First, $\cQ^\dagger \cB$ is $\cQ$-based  $ \frac{\nu} {\beta}$-cocoercive  by Lemma \ref{l_cocoercive}. Then,   $\cI - \cQ^\dagger \cB$ is $\cQ$-based $ \frac {\beta} {2\nu} $-averaged  by \cite[Proposition 4.39]{plc_book}.

\vskip.1cm
(iii) If $\beta>0$, since $\cT = (\cA+\cQ)^{-1} \cQ \circ (\cI - \cQ^\dagger  \cB) $, with $\alpha_1 = \frac{1}{2}$ and $\alpha_2 = \frac {\beta} {2\nu}$, the averagedness of $\cT$, by \cite[Theorem 3]{yamada} or \cite[Proposition 2.4]{plc_yamada}, is given as $\alpha = \frac{\alpha_1+\alpha_2 -2\alpha_1\alpha_2}
{1-\alpha_1\alpha_2} = \frac{2\nu}  {4 \nu  -\beta }$.

If $\beta=0$, $\cT$ becomes $(\cA+\cQ)^{-1}\cQ$, which is $\cQ$-based $\frac{1}{2}$-averaged by the assertion (i).

Both cases are merged as (iii).
\hfill 
\end{proof}

\begin{remark} \label{r_ave}
Proposition \ref{p_com}--(iii) is an extended version of \cite[Proposition 4.14]{pfbs_siam}, from a scalar parameter $\gamma$ to (degenerate)  metric $\cQ$. As observed in \cite[Remark 1]{ljw_mapr}, Proposition \ref{p_com}--(iii)   is sharper than \cite[Proposition 4.32]{plc_book}, which gives the averagedness of $\cT$ as
\[
\alpha = \frac{2} {1+ \frac{1}{\max \{\frac{1}{2},
\frac{2\nu} {\beta} \} } } = 
\frac{2\beta} {\beta + 2\times \min\{\beta, 
\nu \} } = \max \bigg\{ \frac{2}{3},
\frac{2\beta} {\beta+  2\nu } \bigg\}.
\]
\end{remark}

\subsection{Short summary} \label{sec_summary}
This section presents the nonexpansive properties of the G-FBS operator \eqref{gfbs} under degenerate metric setting, which extends the discussions of classical FBS operator in \cite{ljw_mapr}. Most existing works related to metric-based FBS algorithms, e.g., \cite{plc_vu_2014,pesquet_2016,pesquet_2014}, are concerned with the convergence issue under the problem setting of minimization of $f+g$. 
Our exposition here deals with more general operators $\cA$ and $\cB$ (not limited to the subdifferential or gradient of some functions, see further Remark \ref{r_assume_2}-(iv)) under more general metric (in particular, degenerate case). We will see  in Sect. \ref{sec_eg} that these generalizations are essential for the broad applications. 

\section{The G-FBS algorithm under degenerate setting}
\label{sec_gfbs}
We now consider the G-FBS scheme \eqref{gfbs} or its equivalent implicit form \cite[Eq.(54)]{condat_tour}:
\be \label{gfbs_eq}
0 \in \cA x^{k+1} +\cB x^k +\cQ (x^{k+1}-x^k). 
\ee
The convergence has been  investigated in \cite{plc_vu_2014} for the non-degenerate metric. We here focus on the degenerate case.

\subsection{Convergence in terms of metric distance}
\label{sec_dist}
First, it is easy to recognize the following
\begin{fact} \label{f_2}
$\Fix \cT = \zer (\cA + \cB)$.
\end{fact}
\begin{proof}
 Indeed,
$\bbstar \in \Fix \cT \Longleftrightarrow 
\bbstar = (\cA + \cQ)^{-1} (\cQ - \cB) \bbstar
\Longleftrightarrow (\cQ - \cB) \bbstar \in 
 (\cA + \cQ) \bbstar \Longleftrightarrow   0  \in (\cA +\cB) \bbstar \Longleftrightarrow \bbstar \in \zer (\cA +\cB)$. 
 \hfill 
\end{proof}

The following theorem is a main result of this paper, which shows the weak convergence of $\{x^k\}_{k\in\N}$ in $\ran\cQ$. The proof adopts some techniques in \cite[Theorem 2.1]{fxue_gopt}.
\begin{theorem}[Weak convergence in $\ran \cQ$] 
\label{t_dist}
Let $\bb^0\in \cH$, $\{\bb^k\}_{k \in \N}$ be a sequence generated by \eqref{gfbs}.  Under Assumption \ref{assume_1}, if $\nu > \beta/2$,  the following hold.
\begin{itemize}
\item[\rm (i)] {\rm [Finite length in $\ran\cQ$]} $\{\cQ x^k\}_{k\in\N}$ has a finite length in a sense that $\sum_{k=0}^\infty \|x^{k+1}-x^k\|_\cQ^2 < \infty$.

\item[\rm (ii)] {\rm [$\cQ$-based asymptotic regularity]} $\cQ (x^k-x^{k+1}) \rightarrow 0$, as $k\rightarrow \infty$.

\item[\rm (iii)] {\rm [Rate of $\cQ$-based regularity]}  $\|\bb^{k+1 } -\bb^{k} \|_\cQ$ has the pointwise rate of $\cO(1/\sqrt{k})$:
\[
\big\|\bb^{k +1} -\bb^{k} \big\|_\cQ
\le \frac{1}{\sqrt{k+1}}   \sqrt{ \frac{2\nu}{2\nu-\beta} }
\big\|\bb^{0} -\bb^\star \big\|_\cQ, \quad
\forall k \in \N.
\]

\item[\rm (iv)] {\rm [Weak convergence in $\ran \cQ$]} There exists $x^\star \in \zer (\cA+\cB)$, such that $ \cQ  x^k \weak  \cQ x^\star$, as $k\rightarrow \infty$. 
\end{itemize}
\end{theorem}

\vskip.2cm
\begin{proof}
(i)-(ii) Taking $\bb_1 = \bb^k$ and $\bb_2 = x^\star \in \Fix  \cT $ in Lemma \ref{l_T}-(iii),  we obtain
\be \label{x12}
\big\|  \bb^{k+1} -\bbstar \big\|_\cQ^2
  \le   \big\| \bb^k - \bbstar \big\|_\cQ^2
-  \Big( 1-\frac{\beta}{2 \nu }  \Big) 
\big\|  \bb^k -  \bb^{k+1} \big\|_\cQ^2.
\ee
Summing up \eqref{x12} from $k=0$ to $K$ yields
\be \label{x34}
\sum_{k=0}^{K} \big\| \bb^k - \bb^{k+1} \big\|
_\cQ^2 \le  \frac{2\nu}{2\nu-\beta} 
\big\| \bb^0 - \bbstar \big\|_\cQ^2.
\ee 
Taking $K \rightarrow \infty$, we have: $\sum_{k=0}^{\infty} \big\| \bb^k - \bb^{k+1} \big\|_\cQ^2 \le \frac{2\nu}{2\nu-\beta}   \big\| \bb^{0} - \bb^\star \big\|_\cQ^2 < +\infty$, which   implies that  $\lim_{k\rightarrow \infty} \|\bb^k-\bb^{k+1}\|_\cQ =  0$.

\vskip.1cm
(iii)  Taking $\bb_1 = \bb^k$ and  $\bb_2 = \bb^{k+1}$ in Lemma \ref{l_T}-(iii),  we have
\be  \label{x33}
\big\|  \bb^{k+1} -\bb^{k+2} \big\|_\cQ
  \le  \big\| \bb^k - \bb^{k+1}  \big\|_\cQ, 
\ee
which implies that  $\| \bb^k - \bb^{k+1}  \|_\cQ$ is non-increasing. Then, (iii) follows from \eqref{x34}.

\vskip.1cm
(iv) Following the reasoning of  the well-known Opial's lemma \cite{opial}\footnote{Refer to \cite[Lemma 2.47]{plc_book} or \cite[Lemma 2.1]{attouch_2001} for the Opial's argument.}, the weak convergence proof is divided into 3 steps\footnote{This line of reasoning is very similar to Fej\'{e}r monotonicity, see \cite[Proposition 5.4,  Theorem 5.5]{plc_book} for example.}:

{\red
Step-1: show that $\lim_{k\rightarrow \infty} \| x^{k} - x^\star \|_\cQ  $ exists for any given $x^\star \in \zer (\cA+\cB) $; }

Step-2: show that $\{\sqrt{ \cQ} x^k\}_{k\in\N}$ has at least  one weak sequential cluster point lying in $\sqrt{\cQ}  \zer (\cA+\cB) $;

Step-3: show that the  cluster point of  $\{\sqrt{ \cQ} x^k\}_{k\in\N}$ is unique.

\vskip.2cm
Step-1: \eqref{x12}  shows that    the sequence
 $\{ \| x^{k} - x^\star \|_\cQ \}_{k\in\N} $ is non-increasing, and bounded from below (always being non-negative), and thus, convergent, i.e. $\lim_{k\rightarrow \infty} \| x^{k} - x^\star \|_\cQ$ exists.

\vskip.2cm
Step-2: \eqref{x12} also implies that the sequence $\{\sqrt{ \cQ} x^k\}_{k\in\N}$ is bounded, since $\big\|\sqrt{\cQ} x^k- \sqrt{\cQ}\cstar \big\| = \|x^k-\cstar\|_\cQ \le \|x^0 - \cstar\|_\cQ
= \big\|\sqrt{\cQ} x^0- \sqrt{\cQ}\cstar \big\| $, $\forall k \in \N$. Then, by \cite[Lemma 2.37]{plc_book}, $\{\sqrt{ \cQ} x^k\}_{k\in\N}$ has at least  one weak sequential cluster point, i.e. $\{\sqrt{\cQ} x^k\}_{k\in\N}$ has a subsequence $\{\sqrt{ \cQ}  x^{k_i}\}_{i\in\N}$ that weakly converges to a point $v^*$, denoted by $\sqrt{\cQ} x^{k_i} \rightharpoonup v^*$, as $k_i \rightarrow \infty$. Our aim in Step-2  is to show that $v^* = \sqrt{\cQ} x^*$ for some $x^* \in \zer (\cA+\cB)$, and more generally,  every weak sequential cluster point of  $\{ \sqrt{ \cQ} x^k\}_{k\in\N}$ belongs to the set $ \sqrt{\cQ}   \zer  (\cA+\cB) := \{v\in \cH: v=\sqrt{\cQ} x\text{\ and\ } x \in \zer  (\cA+\cB) \}$. To this end, we first note that (i) asserts that   $\cQ (x^k- x^{k+1}) \rightarrow  0$   strongly in $\cH$ as $k \rightarrow \infty$, and further,  $\dist (\cA x^{k+1}+\cB x^k, 0) \rightarrow 0$, as $k\rightarrow \infty$ by the scheme \eqref{gfbs}.  Then, since $\cB$ is $\beta$-Lipschitz continuous (by Remark \ref{r_assume_1}-(ii)), and due to the  sequential closedness of the graph of $\cA+\cB$ in $\cH_\text{weak} \times \cH_\text{strong}$ (by the maximality of $\cA+\cB$ in Remark \ref{r_assume_1}-(ii) and \cite[Proposition 20.38]{plc_book}),  the weak sequential cluster point of $\{x^{k} \}_{k\in \N}$ lies in $\zer (\cA+\cB)$. Owing to the closedness of $\ran \cQ$ (i.e.,  Assumption \ref{assume_1}-(iv)),  $\{\sqrt{ \cQ} x^k\}_{k\in\N}$ has at least  one weak sequential cluster point lying in $\sqrt{\cQ}  \zer  (\cA+\cB)  $.

\vskip.2cm
Step-3: We need to show that   $\{\sqrt{\cQ} x^k\}_{k\in\N}$ cannot have two distinct weak sequential cluster point in $\sqrt{ \cQ}  \zer  (\cA+\cB)$. To this end, let $\sqrt{\cQ} x_1^*, \sqrt{\cQ} x_2^{*} \in \sqrt{\cQ}  \zer  (\cA+\cB)$ be two cluster points of  $\{ \sqrt{ \cQ} x^k\}_{k\in\N}$. Since  $\lim_{k\rightarrow \infty} \| x^{k} - \cstar \|_\cQ $ exists as proved in Step-1, set $l_1 = \lim_{k\rightarrow \infty} \| x^k - x_1^*\|_\cQ = \|\sqrt{\cQ} x^k - \sqrt{\cQ} x_1^*\| $, and $l_2 = \lim_{k\rightarrow \infty} \| x^k - x_2^*\|_\cQ = \|\sqrt{\cQ} x^k - \sqrt{\cQ} x_2^*\| $. Take a subsequence $\{\sqrt{\cQ} x^{k_i} \}$ weakly converging to $ \sqrt{\cQ} x_1^*$, as $k_i \rightarrow \infty$. From the identity of
\[
\big\| x^k-x_1^*\big\|_\cQ^2 - \big\|x^k-x_2^* \big\|_\cQ^2
=\big\|x_1^*- x_2^*\big\|_\cQ^2 +2 \big\langle 
x_1^*- x_2^* \big|  x_2^* - x^k \big\rangle_\cQ,
\]
we deduce that  $l_1 -l_2 =- \big\| x_1^*- x_2^*\big\|_\cQ^2$ by taking $k\rightarrow \infty$ on both sides. Similarly,  take a subsequence $\{ \sqrt{\cQ} x^{k_j} \}$ weakly converging to $\sqrt{\cQ} x_2^*$, as $k_j \rightarrow \infty$, which yields that $l_1 -l_2 = \big\|x_1^* - x_2^*\big\|_\cQ^2$. Consequently, $\big\|x_1^*-x_2^*\big\|_\cQ = 
\big\|\sqrt{\cQ} x_1^* - \sqrt{\cQ} x_2^*\big\| =0$, which establishes the uniqueness of the weak sequential cluster point, denoted by $\sqrt{\cQ} x^\star$. 

Finally, with a trivial replacement of $\sqrt{\cQ}$ by $\cQ$ (see, for instance, \cite[Fact 2.25]{plc_book}), we summarize that  $\{\cQ x^k\}_{k\in\N}$, is bounded and possesses a unique weak sequential cluster point $\cQ x^\star \in \cQ   \zer  (\cA+\cB)$. By \cite[Lemma 2.38]{plc_book}, $\cQ x^k \rightharpoonup \cQ x^\star \in \cQ \zer (\cA+\cB) $, as $k\rightarrow \infty$.
\hfill 
\end{proof}

\vskip.1cm
If  $\cQ$ is non-degenerate, one can safely conclude $x^{k+1}-x^k\rightarrow 0$ and $x^k \weak x^\star$, as $k\rightarrow \infty$. This coincides with  \cite[Theorem 4.1]{plc_vu_2014} and \cite[Theorem 4.8]{repetti}. The convergence results of the basic FBS algorithm with $\cQ = \frac{1}{\tau} \cI$ presented in  \cite{rtr_1976,ppa_guler,corman,taomin_2018} are also exactly recovered.  However, if $\cQ$ is degenerate, one cannot establish the weak convergence of $x^k \rightharpoonup x^\star \in \zer (\cA+\cB)$. It is indeed true that   the weak sequential cluster point of $\{x^{k} \}_{k\in \N}$ lies in $\zer (\cA+\cB)$ as proved in Step-2. However, the cluster point may not be unique. As observed in Step-3, one can only obtain $\big\|x_1^*-x_2^*\big\|_\cQ =0$, which yields $x_1^*- x_2^* \in   \ker \cQ$ rather than $x_1^*=x_2^*$. This is essentially due to the existence of $\lim_{k\rightarrow \infty}\big\|x^k- x^\star\big\|_\cQ$ only (as shown in Step-1), while the existence of $\lim_{k\rightarrow \infty} \big\|x^k- x^\star\big\|$ is not guaranteed. 

In addition, in view of \cite[Lemma 2.7]{corman}, the rate of asymptotic regularity in Theorem \ref{t_dist}-(iii) can be refined to $o(1/\sqrt{k})$.

\vskip.1cm
If $\cB=0$, \eqref{gfbs} becomes the (degenerate) metric PPA:
\be \label{ppa}
x^{k+1} := (\cA+\cQ)^{-1} \cQ x^k,
\ee
whose convergence result  is stated below.
\begin{corollary}[Weak convergence in $\ran \cQ$] 
\label{c_ppa}
Let $\bb^0\in \cH$, $\{\bb^k\}_{k \in \N}$ be a sequence generated by \eqref{ppa}.  Under Assumption \ref{assume_1}  with $\cB=0$,  the following hold.
\begin{itemize}
\item[\rm (i)] {\rm [Rate of $\cQ$-based regularity]}  $\|\bb^{k+1 } -\bb^{k} \|_\cQ$ has the pointwise rate of $\cO(1/\sqrt{k})$:
\[
\big\|\bb^{k +1} -\bb^{k} \big\|_\cQ
\le \frac{1}{\sqrt{k+1}}  
\big\|\bb^{0} -\bb^\star \big\|_\cQ, \quad
\forall k \in \N.
\]

\item[\rm (ii)] {\rm [Weak convergence in $\ran \cQ$]} There exists $x^\star \in \zer \cA$, such that $ \cQ  x^k \weak  \cQ x^\star$, as $k\rightarrow \infty$. 
\end{itemize}
\end{corollary}
\begin{proof}
Substituting $\beta=0$ into Theorem \ref{t_dist}. 
\hfill 
\end{proof}

\subsection{Convergence of objective value}
\subsubsection{Assumptions and basic descent results}
If $\cA = \partial g$, $\cB =\nabla f$, where the functions $f$ and $g$ satisfy the following
\begin{assumption} \label{assume_2}
\begin{itemize}
\item[\rm (i)] $f: \cH \mapsto \R$ and $g: \cH \mapsto \R \cup \{ +\infty\}$ are proper and lower semi-continuous (l.s.c.);

\item[\rm (ii)] $f$ is Fr\'{e}chet differentiable with $\beta$-Lipschitz continuous gradient $\nabla f$;

\item[\rm (iii)] $\dom f\cap \dom g \ne \emptyset$, $\Arg\min (f+g) \ne \emptyset$;

\item[\rm (iv)] $\cQ$ is linear, bounded, closed, self-adjoint and degenerate, such that $\ker\cQ \backslash \{0\} \ne \emptyset$;

\item[\rm (v)] $\ran \nabla f \subseteq \ran \cQ$;

\item[\rm (vi)] $\ran (\partial g+\cQ) \supseteq  \ran (\cQ-\nabla f)$;

\item[\rm (vii)] $g$ is convex;
\item[\rm (viii)] $f$ is convex.
\end{itemize}
\end{assumption}
\begin{remark} \label{r_assume_2}
{\rm (i)} Assumption \ref{assume_2} is parallel to Assumption \ref{assume_1}, by the correspondence of $\cA=\partial g$ and $\cB = \nabla f$. Under Assumption \ref{assume_2}-(iii), the problem \eqref{inclusion} is equivalent to minimizing $f+g$, if $\dom f\cap \dom g \ne \emptyset$, by \cite[Proposition 16.42]{plc_book}.

{\rm (ii)} To understand Assumption \ref{assume_2}-(v), consider a function
$f:\R^2\mapsto \R: x=(a,b)\mapsto \frac{1}{2}a^2$. Then, $\nabla f(x) = \begin{bmatrix}
\partial_a f(x) \\ \partial_b f(x)  \end{bmatrix} = 
\begin{bmatrix}
a \\ 0 \end{bmatrix}$, which lies in the range of a degenerate metric $ \cQ = \begin{bmatrix}
1 & 0 \\ 0 & 0 \end{bmatrix}$, i.e. $\R \times \{0\}$---a proper subspace of $\cH=\R^2$.  Note that here $\cB=\nabla f = \begin{bmatrix}
1 & 0 \\ 0 & 0 \end{bmatrix}$, i.e., the previous example in Remark \ref{r_assume_1}-(iii).
 
{\rm (iii)} The convexity of $f$ and $g$ is assumed in separate items, since some of the following results do not require the convexity.

{\rm (iv)} It should be stressed that Assumption \ref{assume_1} is more general than Assumption \ref{assume_2}. In many examples in Sect. \ref{sec_eg}, $\cA$ is not {\it cyclically} maximally monotone. Consequently, by \cite[Theorem 22.18]{plc_book}, there does not exist a proper, l.s.c. and convex function $f$, such that $\cA = \partial f$. In this sense, the original problem \eqref{inclusion} is essentially beyond the scope of minimizing a certain cost function. This part actually deals with a special case of Sect. \ref{sec_dist}, if there exists an objective function $f+g$ to minimize. 

{\rm (v)} Assumption \ref{assume_2} also covers a special case of $g=0$ or $f=0$, where the G-FBS iteration \eqref{gfbs}, under the case of $\cQ=\frac{1}{\tau}\cI$, becomes classical gradient descent (see Example \ref{eg_grad}) or classical PPA (see Example \ref{eg_ppa} in Sect. \ref{sec_eg}). 
\end{remark}

Let us first show an important descent lemma restricted to the range space of $\cQ$.
\begin{lemma} [Degenerate descent lemma] 
\label{l_descent}
Under Assumption \ref{assume_2}-(ii), (iii) and (iv), it holds that
\[
f(x_2) \le   f(x_1) +  
 \big\langle  \cQ  (x_2-x_1) |  \cQ^\dagger 
\nabla f(x_1) \big\rangle 
 + \frac{\beta} {2\nu}  \big\|  x_2-x_1 \big\|_\cQ^2  ,\ 
 \forall (x_1,x_2) \in \cH \times \cH.
\]
where $\nu$ is specified in Remark \ref{r_assume_1}-(i).
\end{lemma}
\begin{proof}
We adopt similar technique with the proof of \cite[Lemma 2.64]{plc_book}:
\begin{eqnarray}
&& \big| f(x_2) -f(x_1)- 
\big\langle  x_2-x_1 |\nabla f(x_1) \big\rangle \big|
\nonumber \\
& =& \int_0^1 \frac{1}{t} \big\langle t( x_2-x_1) \big|
\nabla f(x_1 +t( x_2-x_1)) - \nabla f(x_1) \big\rangle 
\rd t
\nonumber \\
& \le & \int_0^1 \frac{1}{t} \big\|  t ( x_2-x_1) \big\|_\cQ \cdot \big\| \cQ^\dagger \nabla f(x_1 +t( x_2-x_1)) - \cQ^\dagger \nabla f(x_1)   \big\|_\cQ \rd t
\ \text{[by Fact \ref{f_1}-(ii)]}
\nonumber \\
& \le & \int_0^1 \frac{1}{t} \big\|  t ( x_2-x_1) \big\|_\cQ \cdot \frac{\beta} {\nu}  \big\|  t ( x_2-x_1) \big\|_\cQ  \rd t
\ \text{[by Lemma \ref{l_cocoercive}]}
\nonumber \\
& = &  \frac{\beta} {2 \nu} 
 \big\|  x_2-x_1 \big\|_\cQ^2.
 \nonumber 
\end{eqnarray}
The desired inequality follows by noting that 
$\big\langle  x_2-x_1 |\nabla f(x_1) \big\rangle
= \big\langle  \cQ  (x_2-x_1) |  \cQ^\dagger 
\nabla f(x_1) \big\rangle $, due to Fact \ref{f_1}-(ii). 
\hfill
\end{proof}

\vskip.1cm
The well-known {\it Descent Lemma} \cite[Lemma 2.64, Theorem 18.15-(iii)]{plc_book} gives
\be \label{descent_simple}
f(x_2) \le   f(x_1) +  
 \big\langle x_2-x_1 | \nabla f(x_1) \big\rangle 
 + \frac{\beta} {2}  \big\|  x_2-x_1 \big\|^2  ,\ 
 \forall (x_1,x_2) \in \cH \times \cH.
\ee 
This is instrumental for proving the convergence of $f+g$, e.g., \cite[Lemma 3.1]{pesquet_2014}.  Lemma \ref{l_descent} extends this standard result  to the case of degenerate metric. 

By the proof of Lemma \ref{l_descent}, one can see that, if $\ran \nabla f \subseteq \ran \cQ$ and $x_2-x_1\in\ker\cQ$, 
then, 
\[
\big| f(x_2) -f(x_1)- 
\big\langle  x_2-x_1 |\nabla f(x_1) \big\rangle \big|
= \big| f(x_2) -f(x_1)  \big|\le  \frac{\beta} {2 \nu} 
 \big\|  x_2-x_1 \big\|_\cQ^2  =0, 
\]
i.e., $f(x_1)=f(x_2)$. Here, the first equality comes from $\big\langle  x_2-x_1 |\nabla f(x_1) \big\rangle = 0$, due to $x_2-x_1 \in \ker\cQ = (\ran \cQ)^\perp$ and $\nabla f(x_1) \in \ran \cQ$.  This implies that the directional derivative of the function $f$ at any point $x$ along the direction of $s\in\ker \cQ$ is 0. In other words, by \cite[Definition 17.1]{plc_book},  the directional derivative is given as
\[
f'(x; s) = \inf_{\xi \in \ ]0,+\infty[} \frac{f(x+\xi s) -f(x)}{\xi}
=0,\quad \forall x\in\cH,\ s\in \ker\cQ.
\]
To understand this,  it is helpful to recall the previous example in Remark \ref{r_assume_2}-(ii): $f(x)=f(a,b)=\frac{1}{2}a^2$. It is clear that $f(x_1)=f(x_2)$ for $x_1=(a,b_1)$ and 
$x_2=(a,b_2)$.

The following lemma   extends the {\it sufficient decrease property} \cite[Lemma 2]{bolte_2014} to the case of arbitrary (degenerate) metric $\cQ$.
\begin{lemma} [Degenerate sufficient decrease property] \label{l_decrease}
Let $\bb^0\in \cH$, $\{\bb^k\}_{k \in \N}$ be a sequence generated by \eqref{gfbs}. {\red Define $h:=f+g$.} Then,   the following hold.
\begin{itemize}
\item[\rm (i)] Under Assumption \ref{assume_2}-(i)--(v), we have
\[
h (\bb^k) - h (\bb^{k+1}) \ge
 \frac{1}{2} \Big(1- \frac{\beta} { \nu} \Big)
  \big\| \bb^{k+1} - \bb^k\big\|_\cQ^2;
 \]
 
\item[\rm (ii)] Under Assumption  \ref{assume_2}-(i)--(vi), we have
\[
h  (\bb^k) - h  (\bb^{k+1}) \ge
  \Big(1- \frac{\beta}{2\nu} \Big)
    \big\| \bb^{k+1} - \bb^k\big\|_\cQ^2,
 \]
\end{itemize}
where $\nu$ is specified in Remark \ref{r_assume_1}-(i).
\end{lemma}

\begin{proof}
Rewrite \eqref{gfbs_eq} as
\be \label{gpfbs_eq}
 0  \in \nabla f(\bb^k) +\partial g(\bb^{k+1})
 +\cQ (\bb^{k+1} - \bb^k).
\ee
By  Lemma \ref{l_descent}, we have
\begin{eqnarray}  \label{x1}
h(\bb^{k+1}) & = & f(\bb^{k+1}) + g(\bb^{k+1} )
\nonumber \\
&  \le & f(\bb^k) + \big \langle  \nabla f(\bb^k) | \bb^{k+1} - \bb^k  \big \rangle
+\frac{\beta}{2\nu} \big \|\bb^{k+1} - \bb^k\big \|_\cQ^2
+   g(\bb^{k+1}).
\end{eqnarray}

(i) By the definition of generalized proximity operator, we have
\[
\bb^{k+1} =\arg\min_\bb g(\bb) + \frac{1}{2} \big\|\bb - \bb^k \big\|_\cQ^2 + \big \langle \bb - \bb^k | \nabla f(\bb^k) 
\big \rangle,
\]
which implies
\be \label{x3}
 g(\bb^{k+1}) + \frac{1}{2} \big\|\bb^{k+1} - \bb^k \big\|_\cQ^2 + \big \langle \bb^{k+1} - \bb^k | \nabla f(\bb^k) 
\big \rangle \le g(\bb^k).
\ee
Combining \eqref{x1} with \eqref{x3} yields
\[
h( \bb^{k+1}) \le  
 h(\bb^{k }) - \frac{1}{2}  \Big(1- \frac{\beta}{\nu} \Big)
  \big\| \bb^k - \bb^{k+1} \big\|^2_{\cQ}.
\]

\vskip.2cm
(ii) By convexity of $g$,  we  have
\be \label{x2}
h(\bb^k)  = f(\bb^k) + g(\bb^{k} )  \ge f(\bb^k) 
+   g(\bb^{k+1}) + \big \langle \partial g (\bb^{k+1}) | \bb^k - \bb^{k+1}  \big \rangle.
\ee
Combining \eqref{x1} with \eqref{x2} yields
\begin{eqnarray}
h( \bb^k) - h(\bb^{k+1}) & \ge & - \big \langle
 \nabla f(\bb^{k}) +  \partial g (\bb^{k+1}) | \bb^{k+1} - \bb^k \big \rangle   - \frac{\beta}{2\nu} \big\| \bb^k - \bb^{k+1} \big\|^2_\cQ
\nonumber \\
 & = &  \big  \langle \bb^{k+1} - \bb^k | \bb^{k+1} - \bb^k \big \rangle_\cQ   - \frac{\beta}{2\nu} \big\| \bb^k - \bb^{k+1} \big\|_\cQ^2 \quad \text{[by \eqref{gpfbs_eq}]}
\nonumber \\
 & = & \Big(1- \frac{\beta}{2\nu} \Big) \big\| \bb^{k+1} - \bb^k\big\|_\cQ^2.
  \nonumber 
\end{eqnarray}
\hfill 
\end{proof}

\begin{remark} \label{r_obj}
{\rm (i)} If $\cQ$ is non-degenerate and closed, combining \eqref{descent_simple} with $\|x^{k+1}-x^k\|^2 \le \frac{1}{\nu} \|x^{k+1}-x^k\|_\cQ^2$, one can reach exactly the same result as Lemma \ref{l_decrease}. This similar results can also be found in \cite[Lemma 4.1, Proposition 4.3]{repetti},  \cite[Lemma 4.1]{pesquet_2014} and \cite[Lemma 3.1]{pesquet_2016}.

{\rm (ii)} Lemma \ref{l_decrease} extends the existing results of  sufficient decrease properties to the degenerate setting. More importantly, Lemma \ref{l_decrease} is valid without the convexity of $f$.

{\rm (iii)} Without convexity of $g$ (i.e. Lemma \ref{l_decrease}--(i)), the sufficient decreasing requires   $\nu \ge  \beta $. If $g$ is convex (i.e. Lemma \ref{l_decrease}--(ii)), this condition is relaxed to $\nu \ge \frac{\beta}{2}$. This coincides with the observation in \cite[Remark 4--(iii)]{bolte_2014}. In addition, combining \eqref{gpfbs_eq} with \eqref{x2}, we obtain
\[
g(\bb^{k+1}) + \big\langle \bb^{k+1} - \bb^k | \nabla f(\bb^k) \big \rangle + \big\|\bb^{k+1} - \bb^k \big\|_\cQ^2 \le 
g(\bb^k), 
\]
which is in agreement with the {\it sufficient decrease condition} \cite[Eq.(3.6)]{repetti}, \cite[Eq.(7a)]{pesquet_2014}, \cite[Remark 2.7]{pesquet_2016}. 

{\rm (iv)} If $f=0$, Lemma \ref{l_decrease} reduces to
$ g (\bb^k) - g (\bb^{k+1}) \ge
 \frac{1}{2}   \big\| \bb^{k+1} - \bb^k\big\|_\cQ^2$, which can be further improved as 
 $ g (\bb^k) - g (\bb^{k+1}) \ge
  \big\| \bb^{k+1} - \bb^k\big\|_\cQ^2$, if $g$ is convex.
  If $g=0$, we obtain  
 $ f (\bb^{k+1}) \le f (x^k) -  
  \big\| \bb^{k+1} - \bb^k\big\|_{\cQ-\frac{\beta}{2}\cI}^2$, for which the decrease of $f$ requires $\cQ \succ \frac{\beta}{2}\cI$. 
\end{remark} 

\subsubsection{Convergence result}
The result is given below.
\begin{proposition}[Non-ergodic  rate in terms of objective value] \label{p_gpfbs_obj}
Let $x^\star \in \Arg\min (f+g)$,  $\bb^0\in \cH$, $\{\bb^k\}_{k \in \N}$ be a sequence generated by \eqref{gfbs}. Under Assumption  \ref{assume_2}, if $\nu \ge  \beta $, the objective value $(f+g)(\bb^k)$ converges to $(f+g) (\bbstar)$ with the {\it non-ergodic}  rate of $\cO(1/k)$, i.e.,
\[
   (f+g) (\bb^{k}) - (f+g) (\bbstar)  \le \frac{1}{2k}  
    \big\|  \bb^0 - \bbstar \big\|_\cQ^2.
\]
\end{proposition}

\begin{proof}
Denote  $h := f + g $. By convexity of $f$ and $g$, we have
\begin{eqnarray}  \label{x4}
h( \bbstar) &=& f( \bbstar) + g( \bbstar)  
\nonumber \\
&\ge & f(\bb^{k}) + \big  \langle
 \nabla f (\bb^{k}) | \bb^\star - \bb^{k} \big \rangle
 + g(\bb^{k+1}) + \big \langle \partial g (\bb^{k+1}) | \bb^\star - \bb^{k+1}\big  \rangle.
\end{eqnarray}
Combining \eqref{x4} with \eqref{x1} yields
\begin{eqnarray} \label{x15}
h( \bbstar) - h(\bb^{k+1}) & \ge & - \big \langle
 \nabla f(\bb^{k}) +  \partial g (\bb^{k+1}) | \bb^{k+1} - \bbstar \big \rangle   - \frac{\beta}{2\nu} \big\| \bb^k - \bb^{k+1} \big\|_\cQ^2
\nonumber \\
 & = & \big  \langle \bb^{k+1} - \bb^k | \bb^{k+1} - \bbstar \big \rangle_\cQ   - \frac{\beta}{2\nu} \big\| \bb^k - \bb^{k+1} \big\|_\cQ^2 \quad \text{[by \eqref{gpfbs_eq}]}
\nonumber \\
 & = & \frac{1}{2} \big( 1-\frac{\beta}{\nu} \big) \big\| \bb^{k+1} - \bb^k\big\|_{\cQ}^2 +
 \frac{1}{2} \big\|  \bb^{k+1} - \bbstar \big\|_\cQ^2 
 -\frac{1}{2} \big\|\bb^k -\bbstar \big\|_\cQ^2  
\nonumber \\
 & \ge &  \frac{1}{2} \big\|  \bb^{k+1} - \bbstar \big\|_\cQ^2 
 -\frac{1}{2} \big\|\bb^k -\bbstar \big\|_\cQ^2.  
 \quad \text{[by $\nu \ge \beta$]}
\end{eqnarray}
{\red Adding \eqref{x15}  from $k=0$ to $k=K-1$ and dividing by $K$ yields
\[
\frac{1}{K}\sum_{k=0}^{K-1} h(\bb^k) -   h(\bbstar) 
 \le  \frac{1}{2K} \big\| \bb^{0} - \bbstar \big\|^2_\cQ.
\]
Then, the pointwise rate of $\cO(1/k)$ is obtained, by combining with the fact that $\frac{1}{K}\sum_{k=0}^{K-1} h(\bb^k) \ge h(x^K)$ due to Lemma \ref{l_decrease}--(ii).} The convergence of $h(x^k)$  to $h(x^\star)$ is proved by \cite[Theorem 2.1]{ppa_guler}. \hfill
\end{proof}

\begin{remark} \label{r_obj_2}
{\rm (i)} Proposition \ref{p_gpfbs_obj} extends the standard result of the PFBS algorithm---\cite[Theorem 3.1]{fista}---to arbitrary (degenerate) metric. This non-ergodic rate in terms of cost value was never discussed in variable metric FBS algorithms, e.g., \cite{plc_vu_2014,pesquet_2016,pesquet_2014,repetti}.

{\rm (ii)} In view of \cite[Lemma 2.7]{corman}, the rate can be refined to $(f+g)(x^k) - (f+g)(x^\star) \sim o(1/\sqrt{k})$.

{\rm (iii)} If $f=0$, Proposition \ref{p_gpfbs_obj} boils down to 
$ g (\bb^{k}) - g (\bbstar)  \le \frac{1}{2k}  
    \big\|  \bb^0 - \bbstar \big\|_\cQ^2$. This is an extended result of \cite[Theorem 2.1]{ppa_guler} under arbitrary degenerate metric. If $g=0$, $f (\bb^{k}) - f(\bbstar)  \le \frac{1}{2k}  
    \big\|  \bb^0 - \bbstar \big\|_\cQ^2$. 
\end{remark}

\section{Relaxations of the G-FBS operator}
\label{sec_extension}
\subsection{The  Krasnosel'ski\u{\i}-Mann iteration}
The  Krasnosel'ski\u{\i}-Mann iteration of $\cT$ in  \eqref{gfbs} is given as
\be \label{rgfbs}
\bb^{k+1} := \bb^k + \gamma \big( \cT  x^k  - x^k \big), 
\ee
where $\gamma$ is a relaxation parameter. The scheme \eqref{rgfbs} is also a fixed-pooint iteration of $\cT_\gamma := \cI - \gamma (\cI - \cT)$. 

\begin{fact} \label{f_3}
$\Fix \cT_\gamma  = \Fix \cT = \zer (\cA + \cB)$.  
\end{fact}

The convergence properties of \eqref{rgfbs} are given below. 
\begin{corollary}[Convergence in terms of metric  distance] \label{c_rgfbs}
Let $\bb^0\in \cH$, $\{\bb^k\}_{k \in \N}$ be a sequence generated by \eqref{rgfbs}. Under Assumption \ref{assume_1}, if $\nu >  \beta/ 2 $,  $\gamma \in\ ] 0, 2-\frac{\beta}{2\nu } [$,  then, the following hold.
\begin{itemize}
\item[\rm (i)] {\rm [Weak convergence in $\ran\cQ$]}  There exists $\bbstar \in \zer (\cA +\cB)$, such that $\cQ \bb^k \weak \cQ\bbstar$, as $k\rightarrow \infty$.

\item[\rm (ii)] {\rm [$\cQ$-based asymptotic regularity]}   $\|\bb^{k+1 } -\bb^{k} \|_\cQ$ has the pointwise convergence rate of $\cO(1/\sqrt{k})$:
\[
\big\|\bb^{k +1} -\bb^{k} \big\|_\cQ
\le \frac{1}{\sqrt{k+1}}   \sqrt{ \frac{2\gamma \nu } {(4-2\gamma) \nu  - \beta} } 
\big\|\bb^{0} -\bb^\star \big\|_\cQ, \quad 
\forall k \in \N.
\]
\end{itemize}
\end{corollary}

\begin{proof}
First, $\cT_\gamma $ is $\cQ$-based $\frac{2 \gamma \nu } 
{4 \nu  - \beta }$-averaged, if  $\gamma \in\ ] 0, 2-\frac{\beta}{2\nu} [$, by Theorem \ref{t_T}--(iii). Then, the proof is completed by the similar reasoning with Theorem \ref{t_dist} and Fact \ref{f_3}. \hfill   
\end{proof}

\begin{remark} \label{r_rgfbs}
{\rm (i)} Corollary \ref{c_rgfbs} extends the existing result of the relaxed PFBS \cite[Theorem 25.8]{plc_book} to arbitrary (degenerate) metric $\cQ$, under milder condition.  In \cite[Theorem 25.8]{plc_book}, the condition is $\gamma < \min\{1, \frac{\nu}{\beta} \} +\frac{1}{2}$, which is obtained by the rough estimate of averagedness of $\cT_\gamma$,  given as $\alpha = \gamma \cdot \max\{ \frac{2}{3}, \frac{2\beta}{\beta +2\nu} \}$ (see Remark \ref{r_ave}). By contrast, our result corresponds to the sharper estimate of $\alpha = \frac{2\gamma \nu} 
{4\nu -\beta}$ (i.e. Theorem \ref{t_T}--(iv)).

{\rm (ii)} If $\cQ$ is non-degenerate, the weak convergence of \eqref{rgfbs} holds for the whole space $\cH$, i.e., $x^{k+1}-x^k\rightarrow 0$ and $x^k \weak x^\star$, as $k\rightarrow \infty$.

{\rm (iii)} If $\beta=0$ (i.e., $\cB=0$), \eqref{rgfbs} is a relaxed PPA. The rate of asymptotic regularity becomes
\[
\big\|\bb^{k +1} -\bb^{k} \big\|_\cQ
\le \sqrt{ \frac{\gamma}{ (k+1)( 2-\gamma) } } 
\big\|\bb^{0} -\bb^\star \big\|_\cQ,\quad 
\forall k \in \N. 
\]
\end{remark} 

\subsection{Arbitrary relaxation operator}
In this sequel, we further consider a more general relaxation operator $\cM$:
\be \label{t_relaxed}
\cT_\cM := \cI - \cM (\cI - \cT),
\ee
where $\cT$ is given by \eqref{gfbs}. 
Then, the fixed-point iteration $\bb^{k+1} : = \cT_\cM \bb^k$ can be rewritten as  the following  {\it  relaxed G-FBS}  algorithm:
\be \label{gppa}
\left\lfloor \begin{array}{llll}
 0 & : \in &  \cA \bbtilde^k + \cB \bb^k  +
\cQ (\bbtilde^k - \bb^k), & \text{(FBS step)} \\
\bb^{k+1}   & := & \bb^k +\cM (\bbtilde^k - \bb^k)  .
& \text{(relaxation step)} 
\end{array} \right.
\ee
To the best of our knowledge, \eqref{gppa} has never been discussed before in the literature. The applications of \eqref{gppa} will be illustrated in Section \ref{sec_eg}.

\begin{assumption} \label{assume_3}
\begin{itemize}
\item[\rm (i)] $\cA$ is (set-valued) maximally monotone;

\item [\rm (ii)] $\cB$ is $\beta^{-1}$-cocoercive with $\beta \in  [0,+\infty[$, and $\ran \cB \subseteq \ran \cQ$;

 \item[\rm (iii)] $\cM$ is invertible, i.e., $\cM^{-1}$ exists;

\item[\rm (iv)] $\cS: = \cQ \cM^{-1} \in \bbS_+$, such that $\ker \cS \backslash \{0\} \ne \emptyset$; 

\item[\rm (v)] $\cQtilde: =\frac{1}{2} ( \cQ + \cQ^\top)$ is at least degenerate;

\item[\rm (vi)] $\cQtilde$ is closed, i.e., $\exists \nu \in \ ]0,+\infty[$, such that $\| \cQtilde x\| \ge \nu \|x\|$, $\forall x \in \ran \cQtilde$; 

\item[\rm (vii)] $ \cG := ( 1 - \frac{\beta}{4\nu} ) 
(\cQ +\cQ^\top)  - \cM^\top \cQ \in \bbS_+$;
\item [\rm (viii)] $\ran(\cA+\cQ) \supseteq \ran (\cQ-\cB)$ and $\zer(\cA+\cB) \ne \emptyset$.
\end{itemize}
\end{assumption}

\vskip.1cm
For Assumption \ref{assume_3}-(ii), the case of $\beta=0$ has been discussed in Remark \ref{r_assume_1}-(iv).  The definitions of $\cS$, $\cG$ and $\tilde{\cQ}$ stem from the following Lemma  \ref{l_gppa}, for sake of convenience.  The non-singularity of $\cM$ keeps the basic rationale of the iteration of \eqref{gppa}: $x^{k+1}$ should contain the information of $\xtilde^k$ for the logical update. A simplest case is $\cM=\cI$, which yields $x^{k+1}=\xtilde^k$ and reduces \eqref{gppa} to \eqref{gfbs}.
Other items in Assumption \ref{assume_3} are the same as Assumptions \ref{assume_1} and \ref{assume_2}.

Lemma \ref{l_gppa} presents several key ingredients, which are the `recipe' for proving the convergence of \eqref{gppa}.
\begin{lemma} \label{l_gppa}
Let $\cT$ defined as \eqref{gfbs}. Denote $\cR:=\cI-\cT$.  Let $\bb^\star \in \zer (\cA+\cB)$, $x^0\in \cH$ and $\{\bb^k\}_{k\in\N}$ be a sequence generated by  \eqref{gppa}. Under Assumption \ref{assume_3}, the following hold.
\begin{itemize}
\item[\rm (i)]  $   \big\| \bb^{k+1} - \bbstar \big\|_\cS^2
 \le  \big\| \bb^k - \bbstar \big\|_\cS^2
-   \big\| \bb^k - \bb^{k+1}\big\|_{\cM^{-\top} 
 \cG   \cM^{-1}   }^2 $;

\item[\rm (ii)]  $ \big \langle     \cR \bb^k \big|\cM^\top \cQ (\cR \bb^k - \cR \bb^{k+1} )  \big \rangle \ge  \frac{1}{2}
  \big( 1 - \frac{\beta}{4\nu} \big) 
\big\|   \cR \bb^k - \cR  \bb^{k+1}  \big\|
_{\cQ+\cQ^\top}^2$ ;

\item[\rm (iii)]  $  \big \| \bb^k - \bb^{k+1} \big\|_\cS^2 - 
\big\| \bb^{k+1} - \bb^{k+2} \big\|_\cS^2   
 \ge   \big\|   \cR \bb^k -  \cR \bb^{k+1}  \big\|_\cG^2 $,
\end{itemize}
where $\nu$ is defined in Assumption \ref{assume_3}-(vi).
\end{lemma}

\begin{proof}
(i)  By monotonicity of $\cA$, we develop
\begin{eqnarray} \label{x5}
0 & \le & \langle \cA \bbtilde^k - \cA \bbstar \big| 
\bbtilde^k - \bbstar \rangle 
\nonumber \\
& =& \langle -\cB \bb^k +\cQ(\bb^k - \bbtilde^k) +\cB \bbstar
\big|  \bbtilde^k - \bbstar  \rangle
\quad \text{[by \eqref{gppa} and $x^\star \in \zer (\cA+\cB)$]}
\nonumber \\
& =& \langle  \cQ(\bb^k - \bbtilde^k) \big| \bbtilde^k - \bbstar  \rangle -  \langle  \cB \bb^k - \cB \bbstar \big| \bbtilde^k - \bbstar  \rangle
\nonumber \\
& =& \langle  \cQ \cM^{-1} (\bb^k - \bb^{k+1}) \big| \bbtilde^k - \bbstar  \rangle -  \langle  \cB \bb^k - \cB \bbstar \big| \bbtilde^k - \bbstar  \rangle 
\quad \text{[by \eqref{gppa}]}
\nonumber \\
& =& \langle  \cS (\bb^k - \bb^{k+1}) \big| \bb^k - \bbstar   +\cM^{-1} (\bb^{k+1} - \bb^k) \rangle -  \langle  \cB \bb^k - \cB \bbstar  \big|  \bbtilde^k - \bbstar  \rangle
\quad \text{[by \eqref{gppa}]}
\nonumber \\
& =& \langle  \cS (\bb^k - \bb^{k+1}) \big| \bb^k - \bbstar   \rangle
+  \langle  \cS (\bb^k - \bb^{k+1}), \cM^{-1} (\bb^{k+1} - \bb^k) \rangle  -  \langle  \cB \bb^k - \cB \bbstar \big| \bbtilde^k - \bbstar  \rangle
\nonumber \\
& =& \frac{1}{2} \big\| \bb^k - \bb^{k+1}\big\|^2_\cS
+\frac{1}{2} \big\| \bb^k - \bbstar \big\|_\cS^2
-\frac{1}{2} \big\| \bb^{k+1} - \bbstar \big\|_\cS^2
- \frac{1}{2} \big\| \bb^k - \bb^{k+1}\big\|_{\cM^{-\top} \cS
+\cS \cM^{-1}}^2
\nonumber \\
&  - &  \langle  \cB \bb^k - \cB \bbstar \big| \bbtilde^k - \bbstar  \rangle
\nonumber \\
& =& \frac{1}{2} \big\| \bb^k - \bbstar \big\|_\cS^2
-\frac{1}{2} \big\| \bb^{k+1} - \bbstar \big\|_\cS^2
- \frac{1}{2} \big\| \bb^k - \bb^{k+1}\big\|_{\cM^{-\top} \cS
+\cS \cM^{-1} -\cS }^2
\nonumber \\
&  - &  \langle  \cB \bb^k - \cB \bbstar \big| \bbtilde^k - \bbstar  \rangle. 
\end{eqnarray}
By adopting similar techniques with  \cite[Theorem 1]{lorenz}, the last term of \eqref{x5} becomes
\begin{eqnarray}
&&    -    \langle  \cB \bb^k - \cB \bbstar \big| \bbtilde^k - \bbstar  \rangle
\nonumber \\ 
 & = &    -    \langle  \cB \bb^k - \cB \bbstar \big| \bbtilde^k -
 \bb^k +\bb^k -  \bbstar  \rangle 
\nonumber \\
 & = &    -    \langle  \cB \bb^k - \cB \bbstar \big| \bbtilde^k -
 \bb^k    \rangle 
  -    \langle  \cB \bb^k - \cB \bbstar \big| \bb^k -  \bbstar  \rangle 
\nonumber \\
 & \le  &    -    \langle  \cB \bb^k - \cB \bbstar \big| \bbtilde^k - \bb^k    \rangle 
  -  \frac{1}{ \beta} \big\|  \cB \bb^k - \cB \bbstar \big\|^2 
\nonumber \\
 & \le  & \frac{\nu}{  \beta} \big\| \cQtilde^\dagger \cB \bb^k - \cQtilde^\dagger\cB \bbstar \big\|_{\cQtilde}^2 + \frac{\beta}{4\nu} \big\| \bbtilde^k -  \bb^k    \big\|_{\cQtilde}^2
  -  \frac{\nu}{  \beta} \big\| \cQtilde^\dagger \cB \bb^k -\cQtilde^\dagger \cB \bbstar \big\|_{\cQtilde}^2 
  \ \text{[by Fact \ref{f_2}]}
\nonumber \\
 & =  &  \frac{\beta}{4\nu} \big\| \bbtilde^k -
 \bb^k    \big\|_{\cQtilde}^2 =  \frac{\beta}{8\nu } \big\| \bb^k - \bb^{k+1}    \big\|_{\cM^{-\top} (\cQ+\cQ^\top)\cM^{-1}}^2.
\quad \text{[by \eqref{gppa}]}
 \nonumber 
\end{eqnarray}
By  substituting into \eqref{x5}, it yields
\[
  \frac{1}{2} \big\| \bb^k - \bbstar \big\|_\cS^2
-\frac{1}{2} \big\| \bb^{k+1} - \bbstar \big\|_\cS^2
- \frac{1}{2} \big\| \bb^k - \bb^{k+1}\big\|_{\cM^{-\top} \cG  \cM^{-1}   }^2  \ge 0, 
\]
where $\cG$ is defined in Assumption \ref{assume_3}.

\vskip.2cm
(ii)  By carefully checking the proof of Lemma \ref{l_T}-(i), one can see that Lemma \ref{l_T}--(i) is valid for arbitrary (not necessarily self-adjoint) metric $\cQ$. More specifically, we  have ($\cQ$ here is not necessarily self-adjoint)
\[
0 \le  \langle \cQ(\cR \bb_1 - \cR \bb_2) | \cT\bb_1 - \cT \bb_2 \rangle
- \langle \cB \bb_1 - \cB \bb_2 | \cT\bb_1-\cT\bb_2 \rangle.
\] 
Adding $\|\cR \bb_1 - \cR  \bb_2 \|_\cQ^2$ on both sides, we develop 
\begin{eqnarray}
&& \big\|\cR \bb_1 - \cR  \bb_2 \big\|_\cQ^2 
=\big\|\cR \bb_1 - \cR  \bb_2 \big\|_{\cQtilde}^2  \quad \text{[by definition of $\cQtilde$]}
\nonumber \\ 
& \le & \big \langle  \bb_1 -  \bb_2 \big|\cQ( \cR \bb_1 - \cR \bb_2 ) \big \rangle 
 - \big\langle \cB \bb_1 - \cB \bb_2 \big| \cT\bb_1-\cT\bb_2 
 \big \rangle
\quad \text{[by $\cR + \cT = \cI$]}
\nonumber \\
& = & \big \langle  \bb_1 -  \bb_2 \big| \cQ(\cR \bb_1 - \cR \bb_2 ) \big \rangle 
 + \big\langle \cB \bb_1 - \cB \bb_2 \big| \cR\bb_1-\cR\bb_2 
 \big \rangle
- \big\langle \cB \bb_1 - \cB \bb_2 \big| \bb_1 - \bb_2 
 \big \rangle
\nonumber \\
& \le & \big \langle  \bb_1 -  \bb_2 \big| \cQ(\cR \bb_1 - \cR \bb_2 )\big \rangle 
 + \frac{\nu}{ \beta} \big\| \cQtilde^\dagger \cB \bb_1 -\cQtilde^\dagger \cB \bb_2 \big\|_{\cQtilde}^2 + 
 \frac{\beta}{4\nu} \big\| \cR\bb_1-\cR\bb_2  \big \|_{\cQtilde}^2
 \nonumber \\
&- & \frac{\nu}{  \beta} \big\| \cQtilde^\dagger \cB \bb_1 - \cQtilde^\dagger \cB \bb_2 \big\|_{\cQtilde}^2
\nonumber \\
& = & \big \langle  \bb_1 -  \bb_2 \big|\cQ( \cR \bb_1 - \cR \bb_2 )\big \rangle 
+ \frac{\beta} {4\nu} \big\| \cR\bb_1-\cR\bb_2  \big \|_{\cQtilde}^2,
\nonumber
\end{eqnarray}
which leads to
\begin{eqnarray}
 \big \langle     \bb^k - \bb^{k+1} \big| \cQ( 
 \cR \bb^k - \cR \bb^{k+1} ) \big \rangle  
 & \ge  &   \big( 1 - \frac{\beta}{4\nu} \big) 
\big\|   \cR \bb^k - \cR  \bb^{k+1}  \big\|_{\cQtilde}^2. 
\nonumber 
\end{eqnarray}
Then, (ii) follows from  $\bb^k - \bb^{k+1} = \cM\cR \bb^k$ by \eqref{gppa} and the definition of $\cQtilde$.

\vskip.2cm
(iii) From Lemma \ref{l_gppa}-(ii), we have
\begin{eqnarray}  \label{dd}
& &  \big \| \bb^k - \bb^{k+1} \big\|_\cS^2 - 
\big\| \bb^{k+1} - \bb^{k+2} \big\|_\cS^2   
\nonumber \\
& = &  \big \|\cM \cR \bb^k \big\|_\cS^2 - 
\big\| \cM \cR\bb^{k+1} \big\|_\cS^2  
\quad \text{[by \eqref{gppa}] } 
\nonumber \\
&=& 2 \big \langle  \cR \bb^k \big|  \cM^\top \cS \cM (
 \cR \bb^k - \cR \bb^{k+1} )  \big \rangle   
- \big\| \cR\bb^k - \cR\bb^{k+1} \big\|^2
_{\cM^\top   \cS \cM }  
\nonumber \\ 
& \ge &  \big\|   \cR \bb^k -  \cR \bb^{k+1}  \big\|_{ 
 ( 1 - \frac{\beta}{4\nu} ) 
(\cQ +\cQ^\top)  - \cM^\top \cS\cM }^2.  \qquad 
\text{[by Lemma \ref{l_gppa}-(ii)] }   
\nonumber 
\end{eqnarray}
This completes the proof. \hfill
\end{proof}

\vskip.1cm
In particular, if $\cM=\gamma \cI$, Lemma \ref{l_gppa}-(i) is simplified to
\[
 \big\| \bb^{k+1} - \bbstar \big\|_\cQ^2
 \le  \big\| \bb^k - \bbstar \big\|_\cQ^2
- \frac{1}{\gamma} \big( 2 - \gamma - \frac{\beta}{2\nu} \big) \big\| \bb^k - \bb^{k+1}\big\|_\cQ^2,
\]
which exactly leads to Corollary \ref{c_rgfbs}-(ii).

\vskip.1cm
The following theorem gives   the convergence result.
\begin{theorem}[Convergence in terms of metric distance] \label{t_gppa}
Let $\bb^0 \in \cH$,  $\{\bb^k\}_{k\in\N}$ be a sequence generated by \eqref{gppa}. Under Assumption \ref{assume_3}, if $\exists \eta \in\ ]0, +\infty[$, s.t. $\big \| x^k - x^{k+1}  \big \|
_{\cM^{-\top} \cG \cM^{-1} }^2 \ge \eta \big \| x^k - x^{k+1}  \big \|_\cS^2 $,  the following hold.
\begin{itemize}
\item[\rm (i)] {\rm [Weak convergence in $\ran\cS$]} There exists $\bbstar\in \zer (\cA+\cB)$, such that $\cS \bb^k \weak \cS \bbstar$, as $k \rightarrow \infty$. 

\item[\rm (ii)] {\rm [Rate of $\cS$-asymptotic regularity]}   $\|  \bb^{k } - \bb^{k+1 } \|_\cS$ has the non-ergodic  convergence rate of $\cO(1/\sqrt{k})$, i.e., 
\[
\big\| \bb^{k+1 } - \bb^{k }  \big\|_\cS 
\le \frac{1}{ \sqrt{k+1 } } \frac{1} {\sqrt{\eta}}
\big\|\bb^{0} -\bb^\star \big\|_\cS,
\quad \forall k \in \N.
\] 
\end{itemize}
\end{theorem}

\begin{proof}
(i) Lemma \ref{l_gppa}--(i) becomes
\[
  \big\| \bb^{k+1} - \bbstar \big\|_\cS^2
 \le  \big\| \bb^k - \bbstar \big\|_\cS^2
- \eta  \big\| \bb^k - \bb^{k+1}\big\|_\cS^2,
\]
{\red
which is in spirit the same as \eqref{x12}. By the similar argument of Theorem \ref{t_dist}-(iv), it is easy to prove that:
\begin{itemize}
\item $\lim_{k\rightarrow \infty} \| x^{k} - x^\star \|_\cS$ exists for any given $x^\star \in \zer (\cA+\cB) $; 
\item $\{\sqrt{ \cS} x^k\}_{k\in\N}$ has at least  one weak sequential cluster point lying in $\sqrt{\cS}  \zer (\cA+\cB) $;
\item the  cluster point of  $\{\sqrt{ \cS} x^k\}_{k\in\N}$ is unique.
\end{itemize}
Finally, we summarize that  $\{\cS x^k\}_{k\in\N}$, is bounded and possesses a unique weak sequential cluster point $\cS x^\star \in \cS \zer  (\cA+\cB)$. By \cite[Lemma 2.38]{plc_book}, $\cS x^k \rightharpoonup \cS x^\star \in \cS \zer (\cA+\cB) $, as $k\rightarrow \infty$. }

\vskip.1cm 
(ii) in view  of Lemma \ref{l_gppa}--(i) and (iii), similar to the proof of Theorem \ref{t_dist}-(iii). 
\hfill 
\end{proof}

\begin{remark} \label{r_gppa}
{\rm (i)} The condition of  $\big \| x^k - x^{k+1}  \big \|
_{\cM^{-\top} \cG \cM^{-1} }^2 \ge \eta \big \| x^k - x^{k+1}  \big \|_\cS^2 $ is in general much milder than 
 $ \cM^{-\top} \cG \cM^{-1}  \succeq \eta \cS$. In Example \ref{eg_radmm} of Sect. \ref{sec_eg}, we will see that this condition is satisfied, but   $ \cM^{-\top} \cG \cM^{-1}  \succeq \eta \cS$ is not guaranteed for any $\eta \in \ ]0, +\infty[$.

{\rm (ii)} In particular, if $\cM = \gamma \cI$, $\big \| x^k - x^{k+1}  \big \|
_{\cM^{-\top} \cG \cM^{-1} }^2 \ge \eta \big \| x^k - x^{k+1}  \big \|_\cS^2 $ is equivalent to
 $ \cM^{-\top} \cG \cM^{-1}  \succeq \eta \cS$. This yields that $\frac{1}{\gamma^2} (2-\frac{\beta}{2\nu}) - \frac{1}{\gamma} \ge \frac{\eta} {\gamma}$. One can safely choose the best possible estimate of $\eta = \frac{1}{\gamma} (2-\frac{\beta}{2\nu} - \gamma)$. Thus, Theorem \ref{t_gppa}-(ii) boils down to  Corollary \ref{c_rgfbs}-(ii).

{\rm (iii)} If $\cS$ and $\cG$ are non-degenerate, $\cG$ can be redefined as $\cG: = \cQ+\cQ^\top - \cM^\top \cQ -\frac{\beta}{2}\cI$; the weak convergence of  $x^k \weak x^\star$ is guaranteed without the additional assumption of  $\big \| x^k - x^{k+1}  \big \|
_{\cM^{-\top} \cG \cM^{-1} }^2 \ge \eta \big \| x^k - x^{k+1}  \big \|_\cS^2 $.
\end{remark}

\section{Applications to the first-order operator splitting algorithms}
\label{sec_eg}
This part  shows that a great variety of operator splitting algorithms falls into the G-FBS category. More importantly, we show that all the properties of each algorithm can be readily obtained from our general results in Sect. \ref{sec_gfbs} and \ref{sec_extension}.

\subsection{The  ADMM/DRS algorithms}
ADMM is one of the most commonly used algorithms for solving the  structured constrained optimization \cite{boyd_admm}:
\be \label{problem1}
\min_{u,v} f(u) +g(v),\quad
\text{s.t.}\ \ Au+Bv  = c,
\ee 
where  $u \in \cU$, $v \in \cV$, the operators $A: \cU  \mapsto \cZ$ and $B: \cV \mapsto \cZ$ are linear and bounded. The functions   $f: \cU \mapsto \R\cup\{+\infty\}$ and  $g: \cV \mapsto \R\cup\{+\infty\}$ are proper, l.s.c.  and convex.
Two typical ADMM  algorithms are listed below.

\begin{example} [Relaxed-ADMM] \label{eg_radmm}
The relaxed-ADMM, or equivalent  relaxed-DRS applied to the dual problem \cite{self_eq}, is given as  \cite[Eq.(3)]{fang_2015}
\be \label{radmm}
\left\lfloor \begin{array}{lll}
u^{k+1} & :\in &  \Arg \min_u   f(u) +
\frac{\tau}{2} \big\| Au +Bv^k - c - \frac{1}{\tau} p^k  \big\|^2,  \\
v^{k+1} & :\in & \Arg \min_v   g(v) + \frac{\tau}{2} \big\| B (v - v^k) +  \gamma (A u^{k+1}
+  B v^{k} - c ) - \frac{1}{\tau}  p^k \big\|^2, \\
p^{k+1} & := & p^k -\tau B (v^{k+1} - v^k) 
  - \tau\gamma (Au^{k+1} + Bv^{k} - c) ,
\end{array} \right. 
\ee
which fits into the relaxed G-FBS operator \eqref{t_relaxed} as:
\[
\bb^k =   \begin{bmatrix}
u^k  \\ v^k \\  p^k  \end{bmatrix},\ 
 \cA = \begin{bmatrix}
\partial f &  0 & -A^\top \\
 0 & \partial g &   -B^\top \\
A  & B & \partial l  \end{bmatrix} , \  
\cB=0,
\]
\[
\cQ =   \begin{bmatrix}
 0 &  0 &  0 \\
 0 & \tau B^\top B   & (1-\gamma) B^\top \\
 0 & -B  & \frac{1}{\tau} I
\end{bmatrix},\  
\cM =  \begin{bmatrix}
I &  0 &  0 \\   0 & I &  0 \\
  0 &  -\tau B  & \gamma  I  \end{bmatrix}, 
\]
where the function $l$ is $l = -\langle \cdot|c \rangle$.
\end{example} 

\begin{proposition} \label{p_radmm}
Let $\{(u^k,v^k,p^k)\}_{k\in\N}$ be a sequence generated by \eqref{radmm}. If $\tau \in\ ]0,+\infty[$ and $\gamma \in \ ]0,2[$,   the following hold.
\begin{itemize}
\item[\rm (i)]  $\big\| \xbar^{k+1 } - \xbar^{k }  \big\|_\cSbar 
\le  \sqrt{ \frac{\gamma} {(2-\gamma)(k+1)} }
\big\|\xbar^{0} - \xbar^\star \big\|_\cSbar$,
$\forall k \in \N$,  where $\xbar = (v,p)$, and $ \cSbar =
  \begin{bmatrix}
  \tau  B^\top B   & (1- \gamma) B^\top  \\
 (1-\gamma) B   & \frac{1}{\tau } I 
\end{bmatrix}$.

\item[\rm (ii)]  There exists a solution $(u^\star, v^\star, p^\star)$ to the problem \eqref{problem1}, such that $(A u^k, B v^k, p^k) \weak  ( Au^\star, Bv^\star, p^\star)$,  as $k \rightarrow \infty$. 
\end{itemize}
\end{proposition}
\begin{proof}
To apply Theorem \ref{t_gppa}, let us check if Assumption \ref{assume_3} is satisfied. First, it is easy to see that 
$\cA$ is maximally monotone,
 $\cB = 0$ is $0$-Lipschitz continuous (i.e., $\beta=0$), and $\cM^{-1}$ exists.  We then compute $\cS$ and $\cG$ as
\[
\cS = \frac{1}{\gamma}  \begin{bmatrix}
   0  &  0 &  0  \\
 0   & \tau  B^\top B   & (1- \gamma) B^\top  \\
  0    & (1-\gamma) B   & \frac{1}{\tau } I 
\end{bmatrix},\  \cG = \begin{bmatrix}
  0 &  0 &   0  \\   0 &  0 &  0  \\
  0 &   0  & \frac{2-\gamma}{\tau} I 
\end{bmatrix},
\]
which are positive semi-definite, if $\tau \in\ ]0,+\infty[$ and $\gamma \in \ ]0,2[$.

Now, we check if  $\big \| x^k - x^{k+1}  \big \|
_{\cM^{-\top} \cG \cM^{-1} }^2 \ge \eta \big \| x^k - x^{k+1}  \big \|_\cS^2 $ holds for some  $\eta \in\ ]0,+\infty[$.  The operator $ \cM^{-\top}  \cG \cM^{-1} - \eta \cS$ is given as
\[
 \cM^{-\top}  \cG \cM^{-1} - \eta \cS  =   \begin{bmatrix}
  0 &   0 &  0 \\
 0 & (\frac{\tau (2-\gamma)}{\gamma^2} - \frac{\tau \eta}{\gamma} ) B^\top B   
&  (\frac{ 2-\gamma }{\gamma^2} - \eta \frac{1-\gamma }{\gamma} ) B^\top \\
  0 & (\frac{ 2-\gamma }{\gamma^2} - \eta \frac{1-\gamma }{\gamma} ) B   &  (\frac{ 2-\gamma } {\tau \gamma^2} - \frac{ \eta} {\tau \gamma} )   I 
\end{bmatrix},
\]
which can be decomposed as $\cS_1 +\cS_2$, where
\[
\cS_1  =  \frac{2 - (\eta +1)\gamma } {\gamma^2} \begin{bmatrix}
  0 &   0 &  0 \\
 0 & \tau B^\top B    & B^\top \\
  0 & B   & \frac{ 1}{\tau}  I 
\end{bmatrix},\quad 
\cS_2  =   \eta \begin{bmatrix}
  0 &   0 &  0 \\   0 & 0 &   B^\top \\
  0 &   B   & 0  \end{bmatrix}.
\]
It is easy to verify that $\cS_1 \in \bbS_+$ for $\gamma \in \ ]0,2[$ and $\eta \in \ ]0, \frac{2}{\gamma} - 1]$, since $\|x\|^2_{\cS_1} = \frac{2 - (\eta+1) \gamma} {\gamma^2} \big\|\sqrt{\tau} Bv +\frac{1}{\sqrt{\tau} } p \big\|^2 \ge 0$, $\forall x=(u,v,p)$. 

What remains to prove is that $ \big \| x^k - x^{k+1}  \big \|^2_{ \cS_2 } \ge 0$, i.e., $\langle B(v^k-v^{k+1}) | p^k-p^{k+1}\rangle \ge 0$. To see this, we consider the convexity of $g$:
\[
g(v) \ge g(v^{k+1}) + \langle \partial g(v^{k+1}) |v-v^{k+1} \rangle,\quad 
g(v) \ge g(v^{k}) + \langle \partial g(v^{k}) | v-v^{k} \rangle.
\]
By the $v$-update  of \eqref{radmm}, we have $B^\top p^{k+1} \in \partial g(v^{k+1})$, and thus,  $B^\top p^{k} \in \partial g(v^{k})$. Then, the above inequalities become, $\forall p \in \mathcal{Z}$:
\[
g(v) \ge g(v^{k+1}) + \langle B^\top p^{k+1}| v-v^{k+1} \rangle,\quad 
g(v) \ge g(v^{k}) + \langle B^\top p^{k} |v-v^{k}  \rangle.
\]
Taking $v=v^k$ in the first inequality, $v=v^{k+1}$ in the second, and summing up both yields  $\langle p^k-p^{k+1}| B(v^k-v^{k+1})\rangle \ge 0$.   Thus, we prove that $\big \| x^k - x^{k+1}  \big \|
_{\cM^{-\top} \cG \cM^{-1} }^2 \ge \eta \big \| x^k - x^{k+1}  \big \|_\cS^2 $ holds for $\eta\in \ ]0,\frac{2}{\gamma}-1]$, if $\gamma \in\ ]0,2[$.

{\rm (i)} In view of Theorem Theorem \ref{t_gppa}-(ii) and choosing $\eta =\frac{2}{\gamma} -1$. This shows  the strong convergences of  $Bv^{k+1}-Bv^k \rightarrow 0$  and $p^{k+1}-p^k \rightarrow 0$. 
 
{\rm (ii)} By Theorem \ref{t_gppa}-(i), we obtain $Bv^k\weak Bv^\star$ and $p^k \weak p^\star$. Then, $Au^k \weak Au^\star$ follows from the $p$-update of \eqref{radmm} and  the strong convergences of  $Bv^{k+1}-Bv^k \rightarrow 0$  and $p^{k+1}-p^k \rightarrow 0$. 
Finally, note that $(u^\star,v^\star,p^\star)\in \zer \cA$ satisfies the Karush-Kuhn-Tucker conditions of \eqref{problem1}. \hfill
\end{proof} 

\vskip.1cm
The scheme \eqref{radmm} is also known as the relaxed-DRS, and  the standard DRS/ADMM is exactly recovered by letting $\gamma=1$ \cite{self_eq}.  Since $\cS$ is degenerate, one can only conclude the strong convergences of $B(v^{k+1}-v^k) \rightarrow 0$ and $p^{k+1}-p^k \rightarrow 0$, but except for $A(u^{k+1}-u^k) \rightarrow 0$. Furthermore, one cannot conclude  $u^k \weak u^\star$ and $v^k \weak v^\star$. Actually, $u^k$ and $v^k$ may not be uniquely determined by \eqref{radmm}.  Hence, we use `$:\in \Arg\min$' instead of `$:= \arg\min$' in the $(u,v)$-updates.  This is related to the notion of {\it infimal postcomposition} \cite{arias_infimal,self_eq}, which is beyond the scope of this paper, and not discussed in details here.

Historically, the weak convergence of $p^k \weak p^\star$ has long been proved in the seminal work of \cite{lions}, while the weak convergence of $Au^k\weak Au^\star$ (i.e., in terms of the solution itself) was recently settled in \cite{svaiter}. Our proof addresses this intricate problem in a much easier way and  under more general problem setting of \eqref{problem1}. 

Finally, we stress that the relaxed-ADMM \eqref{radmm} is a good example to show that  $\big \| x^k - x^{k+1}  \big \|
_{\cM^{-\top} \cG \cM^{-1} }^2 \ge \eta \big \| x^k - x^{k+1}  \big \|_\cS^2 $, but  $ \cM^{-\top} \cG \cM^{-1}  \succeq \eta \cS$ does not hold for any $\eta \in \ ]0, +\infty[$.

\begin{example} [Proximal-ADMM] \label{eg_padmm}
The proximal-ADMM is given as  \cite[modified SPADMM]{lxd_2016}\footnote{The solutions to $u$ and $v$-steps in \eqref{padmm} are guaranteed to be unique, if $P_1, P_2 \in \bbS_{++}$. Hence, we use `$:= \arg\min$' here.} 
\be \label{padmm}
\left\lfloor \begin{array}{lll}
u^{k+1} & := &  \arg \min_u   f(u) +
\frac{\tau}{2} \big\| Au +Bv^k - c - \frac{1}{\tau} p^k \big\|^2  + \frac{1}{2} \big\| u - u^k    \big\|_{P_1}^2 , \\
v^{k+1} & := & \arg \min_v   g(v) + \frac{\tau}{2} \big\|A u^{k+1} + Bv - c - \frac{1}{\tau} p^k \big\|^2 
 + \frac{1}{2} \big\| v - v^k  \big\|_{P_2}^2 , \\
p^{k+1} & := & p^k -\tau (A u^{k+1} + Bv^{k+1}  - c) ,
\end{array} \right. 
\ee
It fits into  the relaxed G-FBS operator \eqref{t_relaxed} as:
\[
\cQ =   \begin{bmatrix}
P_1 &  0 &  0 \\
 0 & P_2 + \tau B^\top B   &  0  \\
 0 & -B  & \frac{1}{\tau} I
\end{bmatrix},\ 
\cM =  \begin{bmatrix}
I &  0 &  0 \\  0 & I &  0 \\
 0 &  -\tau B  & I  \end{bmatrix}, 
\]
where  $x^k$, $\cA$ and $\cB$ are same as Example \ref{eg_radmm}.
\end{example}

\begin{proposition} \label{p_padmm}
Let $\{(u^k,v^k,p^k)\}_{k\in\N}$ be a sequence generated by \eqref{padmm}. If $P_1,P_2 \in\bbS_{++}$ and $\tau >0$, 
 then, there exists a solution $(u^\star, v^\star, p^\star) $ to \eqref{problem1}, such that $(u^k, v^k, p^k) \weak  ( u^\star, v^\star, p^\star)$,  as $k \rightarrow \infty$. 
\end{proposition}

\begin{proof}
The proof is similar to  Proposition \ref{p_radmm}.  We compute $\cS$ and $\cG$ as
\[
\cS =   \begin{bmatrix}
  P_1  &  0 &  0  \\
 0   & P_2 + \tau  B^\top B   & 0  \\
  0    & 0   & \frac{1}{\tau } I 
\end{bmatrix},\  \cG = \begin{bmatrix}
 P_1 &  0 &   0  \\   0 &  P_2 &  0  \\
  0 &   0  & \frac{1}{\tau} I 
\end{bmatrix} .
\]
Then, the proof is completed by Theorem \ref{t_gppa} and Remark \ref{r_gppa}-(iii).
\hfill 
\end{proof}

\vskip.1cm
Comparing Examples \ref{eg_radmm} and \ref{eg_padmm}, one can see that the preconditioning technique of using $P_1$ and $P_2$ basically changes the algorithmic structure. By the proximization,  $(u^k,v^k)$ are uniquely determined in \eqref{padmm}. The corresponding $\cS$ and $\cG$ become non-degenerate, and thus, the weak convergence of all variables $(u^k,v^k,p^k)$ is guaranteed.

\subsection{Gradient descent, P-FBS and PDS algorithms}
\label{sec_grad}
Consider the primal problem \cite[Problem 4.1]{vu_2013}:
\be \label{p}
\min_x f(x)   + \sum_{i=1}^m 
 (g_i \square l_i ) (A_i x  - r_i) 
+ h(x) + \langle x | z \rangle, 
\ee
where $x \in \cX$, $A_i: \cX\mapsto \cY_i $, $r_i \in \cY_i$, $z \in \cX$.  The functions are defined as  $f: \cX \mapsto \R$, $g_i: \cY_i \mapsto \R\cup\{+\infty\}$, $l_i: \cY_i  \mapsto \R\cup\{+\infty\}$, $h: \cX\mapsto \R\cup\{+\infty\}$.  {\red The symbol $\square$ denotes the infimal convolution of both functions $g_i$ and $l_i$, which is defined by
$(g_i \square l_i)(x) = \inf_{u \in \cY_i} \{ g_i (u) + l_i(x-u)\}$, $\forall x\in \cY_i$.} We assume the functions in \eqref{p} satisfy
{\red
\begin{assumption} \label{assume_4}
\begin{itemize}
\item[\rm (i)] $f$, $g_i$, $l_i$ and $h$ are proper, l.s.c. and convex for $i=1,...,m$;

\item [\rm (ii)] $f$ is differentiable  with   $\beta$-Lipschitz continuous gradient; 

 \item[\rm (iii)] $l_i$ is $\mu_i$-strongly convex for $i=1,...,m$.
\end{itemize}
\end{assumption} }

There are various classes of algorithms for solving \eqref{p} or the special cases, listed below. Note that for all algorithms in Sect. \ref{sec_grad}, $\beta$ always stands for the Lipschitz constant of $\nabla f$. 

\begin{example} [Gradient descent]
\label{eg_grad}
Consider $\min_x f(x)$, which is a special case of \eqref{p} with $m=1$, $l=\iota_{ \{ 0\}  }$ (i.e., the indicator function of the set $C=\{ 0\}$), $g= 0$, $h=0$,  $ z =   0$.  If $\Arg \min f\ne \emptyset$, the gradient descent method is given by \cite[Sect. 1.2.1]{bert_book_nonlinear}
\be \label{grad}
x^{k+1} :=  x^k - \tau  \nabla f( x^k),
\ee
which fits into the G-FBS operator \eqref{gfbs} with 
$\cA = 0$, $\cB = \nabla f$ and $\cQ = \frac{1}{\tau} \cI$.
\end{example}

\begin{proposition} \label{p_grad}
Let $x^\star\in\Arg\min f$, $x^0\in \cH$, $\{x^k\}_{k\in\N}$ be a sequence generated by \eqref{grad}. {\red Under Assumption \ref{assume_4}}, if  $\tau \in\ ]0, 2/\beta[$,   the following hold.
\begin{itemize}
\item[\rm (i)] $f(x^{k+1}) \le f(x^k) - \big( \frac{1}{\tau} - \frac{\beta}{2}\big) \big\|x^{k+1}-x^k \big\|^2$.

\item[\rm (ii)] $f(x^{k}) -f(x^\star) \le \frac{1}{2k\tau} 
 \big\|x^{0}-x^\star \big\|^2$, and $f(x^k) \downarrow f(x^\star)$.

\item[\rm (iii)] $f(x^{k}) -f(x^\star) \sim o(1/k)$.

\item[\rm (iv)] $ \big\| \bb^{k+1 } - \bb^{k }  \big\| 
\le \sqrt{ \frac{2}{ (k+1 ) (2-\tau\beta) } } 
\big\|\bb^{0} -\bb^\star \big\|$.

\item[\rm (v)]  $x^k\weak x^\star$.

\end{itemize}
\end{proposition}

\begin{proof}
{\rm (i)} Lemma \ref{l_decrease}-(ii) and Remark \ref{r_obj}-(iv).

{\rm (ii)} Proposition  \ref{p_gpfbs_obj} and Remark \ref{r_obj_2}-(iii).

{\rm (iii)} In view of (ii) and \cite[Lemma 2.7]{corman}. 

{\rm (iv)} Substituting $\nu=1/\tau$ into Theorem \ref{t_dist}-(iii).

{\rm (v)} Theorem \ref{t_dist}-(iv). \hfill 
\end{proof}

\vskip.1cm
Proposition \ref{p_grad} exactly recovers the classical convergence condition of $\tau \in\ ]0, 2/\beta[$, which coincides with  \cite[Proposition 63]{plc_fixed} and \cite[Propositions 1.2.2 and 1.3.3]{bert_book_nonlinear}.

\begin{example} [Classical PPA]
\label{eg_ppa}
Consider  $\min_x g(x)$, which is a special case of \eqref{p} with $m=1$, $l=\iota_{ \{ 0\}  }$, $f=0$, $h=0$,  $z =  0$. Assuming $\Arg\min g \ne \emptyset$,  the classical PPA is given by \cite[Proposition 64]{plc_fixed}
\be \label{classic_ppa}
x^{k+1} := x^k +\gamma  \big( \prox_{\tau g}(x^k) - x^k \big),
\ee
which fits the relaxed G-FBS operator \eqref{rgfbs} with 
$\cA=\partial h$, $\cB=0$ and $ \cQ = \frac{1}{\tau} \cI$. 
\end{example}

\begin{proposition} \label{p_ppa}
Let $x^0\in\cH$, $\{x^k\}_{k\in\N}$ be a sequence generated by \eqref{classic_ppa}. {\red Under Assumption \ref{assume_4}}, if $\tau >0$ and $\gamma \in\ ]0,2[$,   the following hold.
\begin{itemize}
\item[\rm (i)] $x^k\weak x^\star \in \Arg\min g$.

\item[\rm (ii)] $\big\| \bb^{k+1 } - \bb^{k }  \big\| 
\le \sqrt{ \frac{\gamma}{ (k+1)(2-\gamma) } }
\big\|\bb^{0} -\bb^\star \big\|$, and thus, $x^{k+1}-x^k \rightarrow 0$ as $k \rightarrow \infty$.

\item[\rm (iii)]  If $\gamma=1$, $g(x^k) -g(x^\star) \le \frac{1}{2\tau k}
 \big\|\bb^{0} -\bb^\star \big\|^2$, $\forall k \in \N$.
\item[\rm (iv)]  If $\gamma=1$, $g(x^k) -g(x^\star) \sim  o(1/k)$.
\end{itemize}
\end{proposition}

\begin{proof}
{\rm (i)--(ii)} In view of Remark  \ref{r_rgfbs}-(iii) or 
substituting $\nu=\frac{1}{\tau}$ and $\beta=0$ into  Corollary  \ref{c_rgfbs}. 

{\rm (iii)} Proposition \ref{p_gpfbs_obj} or Remark \ref{r_obj_2}-(iii). 

{\rm (iv)} Remark  \ref{r_obj_2}-(ii). 
\hfill 
\end{proof} 

\vskip.1cm
Many classical results presented in the seminal work \cite{ppa_guler} can be retrieved: Proposition \ref{p_ppa}-(i) and (iii) recovers \cite[Theorem 2.1]{ppa_guler}; (iv) recovers \cite[Theorem 3.1]{ppa_guler}; (ii) extends \cite[Corollary 2.3]{ppa_guler} to any relaxation parameter $\eta \in\ ]0,2[$.

\begin{example} [Classical proxmal FBS \cite{plc,plc_chapter}]
Consider the problem $\min_x f(x) + g(x)$, which is a special case of \eqref{p} with $m=1$, $l=\iota_{ \{0\}  }$, $h= 0$,  $z = 0$. Assuming $0\in \sri(\dom g- \dom f)$,  the error-free version of the classical proximal FBS is given by \cite[Eq.(3.6)]{plc}:
\be \label{pfbs}
x^{k+1} := x^k + \gamma  \big( \prox_{\tau g}(x^k - \tau \nabla f(x^k)) - x^k \big),
\ee
which fits the relaxed G-FBS operator \eqref{rgfbs} with 
$\cA=\partial g$, $\cB = \nabla f$ and  $\cQ = \frac{1}{\tau} \cI$.
\end{example}

\begin{proposition} \label{p_pfbs}
Let $x^0 \in \cH$,  $\{x^k\}_{k\in\N}$ be a sequence generated by \eqref{pfbs}. {\red Under Assumption \ref{assume_4}}, if $\tau> \beta/2$ and $\gamma \in\ ]0,2[$,  the following hold.
\begin{itemize}
\item[\rm (i)] $x^k\weak x^\star \in \Arg\min (f+g)$.

\item[\rm (ii)] $\big\| \bb^{k+1 } - \bb^{k }  \big\|
\le \frac{1}{ \sqrt{k+1 } } 
\sqrt{\frac{2\gamma} {(4-2\gamma)-\tau \beta } }
\big\|\bb^{0} -\bb^\star \big\|$,  $\forall k \in \N$.

\item[\rm (iii)] If $\gamma=1$, $(f+g)(x^k) - (f+g)(x^\star) \le \frac{1}{2\tau k}
 \big\|\bb^{0} -\bb^\star \big\|^2$, $\forall k \in \N$.
 
\item[\rm (iv)]  If $\gamma=1$, $(f+g)(x^k) - (f+g)(x^\star) \sim  o(1/k)$.
\end{itemize}
\end{proposition}

\begin{proof}
{\rm (i)--(ii)} Substituting $\nu=\frac{1}{\tau}$ into Corollary \ref{c_rgfbs}.

{\rm (iii)} Proposition \ref{p_gpfbs_obj}. 

{\rm (iv)} Remark  \ref{r_obj_2}-(ii). 
 \hfill
\end{proof}

\vskip.1cm
This algorithm \eqref{classic_ppa} is also known as the {\it proximal gradient method} \cite{boyd_prox}.  Proposition \ref{p_pfbs} loosens the convergence condition in  \cite[Theorem 3.4]{plc} and \cite[Proposition 10.4]{plc_chapter} from  $\gamma \in \ ]0,1]$ to $\gamma \in\ ]0, 2[$.

\begin{example} [Chambolle-Pock algorithm \cite{cp_2011}] \label{eg_cp}
Consider $\min_u h(u) + g(Au)$, which is a special case of \eqref{p} with $m=1$, $l=\iota_{ \{  0\}  } $, $f=  0$, $r =  0$, $z =  0$.  Assuming $0\in\sri(\dom g -A(\dom h))$,  the Chambolle-Pock algorithm is \cite[Algorithm 1]{cp_2011}:
\be \label{cp}
\left\lfloor \begin{array}{llll}
s^{k+1}   & := &  \prox_{\sigma g^*}  \big(s^k +\sigma 
A (2u^k - u^{k-1} )  \big), & \text{\rm (dual step)} \\
u^{k+1}   & := &  \prox_{\tau h}  \big( u^k - \tau   
A^\top s^{k+1}   \big). &  \text{\rm (primal step)}   
\end{array} \right. 
\ee
The corresponding G-FBS operator \eqref{gfbs} is:
\[
x^k = \begin{bmatrix}  s^{k} \\ u^{k-1}
\end{bmatrix}, \  
\cA=   \begin{bmatrix}
\partial g^* & -A \\
A^\top  & \partial h \end{bmatrix},\  
\cB=0, \  \cQ =  \begin{bmatrix}
\frac{1}{ \sigma} I & - A \\
- A^\top   & \frac{1}{ \tau} I 
\end{bmatrix} .
\]
\end{example}
In the above  G-FBS fitting,  we use a mismatch of iteration indices between $u$ and $s$: $x^k := (s^k, u^{k-1})$. This technique can also be found in \cite{bot_2015}.

\begin{proposition}
Let $\{ (s^k, u^{k})\}_{k\in\N}$ be a sequence generated by \eqref{cp}. {\red Under Assumption \ref{assume_4}}, if  $\tau \in\ ]0, \frac{1}{\sigma \|A^\top A\|}[$,  the following hold.
\begin{itemize}
\item[\rm (i)] $ (s^k, u^{k}) \weak (s^\star,u^\star) \in \zer \cA$; 

\item[\rm (ii)] $\big\| \bb^{k+1 } - \bb^{k }  \big\|_\cQ 
\le \frac{1}{ \sqrt{k+1 } } 
\big\|\bb^{0} -\bb^\star \big\|_\cQ$, 
$ \forall k \in \N$.
\end{itemize}
\end{proposition}
\begin{proof}
In view of Corollary \ref{c_ppa} and \cite[Theorem 16.47]{plc_book}. \hfill 
\end{proof}

\vskip.1cm
Comparing with \cite[Theorem 1]{cp_2011}, which only depends on the diagonal part of $\cQ$,  our result takes into account the off-diagonal components of $\cQ$. This is consistent with the result of \cite{cp_2016}---an improved version of \cite{cp_2011}.

\begin{example} [Arias-Combettes algorithm \cite{arias_2011}]
Consider the  problem 
$\min_u h(u) + g(Au -r) + \langle u| z\rangle$, which is a special case of \eqref{p} with $m=1$, $l=\iota_{ \{ 0\}  } $, $f= 0$.  Assuming $\Arg\min (h+g\circ (A\cdot -r) - \langle \cdot|z\rangle)\ne \emptyset$ and $r \in \sri(A(\dom h)-\dom g)$,  the error-free version of \cite[Proposition 4.2]{arias_2011} is given as
\be \label{arias}
\left\lfloor \begin{array}{lll}
\utilde^{k}  & := & \prox_{\tau h} \big( u^k - \tau
A^\top s^k - \tau z  \big), \\
\stilde^{k}  & := & \prox_{\tau  g^*} \big( s^k +\tau A u^k  -\tau r \big), \\
u^{k+1}  & := & \utilde^k - \tau A^\top (\stilde^k - s^k), \\
s^{k+1}  & := & \stilde^k +\tau A (\utilde^k - u^k)  ,
\end{array} \right.
\ee
which corresponds to the following relaxed G-FBS operator:
\[
x^k =\begin{bmatrix} u^{k } \\ s^{k }\end{bmatrix},
\    \cA = \begin{bmatrix}
\partial \ftilde   & A^\top \\
- A & \partial \gtilde^*   \end{bmatrix},\   
\cB= 0, \   \cQ =  \begin{bmatrix}
\frac{1}{\tau} I & -A^\top  \\
A    & \frac{1}{\tau} I \end{bmatrix}, \  
\cM =  \begin{bmatrix}
 I & -\tau A^\top   \\  \tau A   &  I 
\end{bmatrix},
\] 
where $\ftilde = f+\langle \cdot | z\rangle$, 
$\gtilde^* = g^* +\langle \cdot | r\rangle$.  
\end{example}

\begin{proposition}
Let $\{ (s^k, u^{k})\}_{k\in\N}$ be a sequence generated by \eqref{arias}.  {\red Under Assumption \ref{assume_4}}, if $\tau \in\ ]0,  \frac{1}{ \|A\| }[$, then,  $ (s^k, u^{k}) \weak (s^\star,u^\star) \in \zer \cA$.
\end{proposition}
\begin{proof}
We compute $\cS$ and $\cG$ as:
\[
\cS =  \begin{bmatrix}
   \frac{1}{ \tau} I    & 0  \\
  0  & \frac{1}{\tau } I \end{bmatrix},\  \cG = \begin{bmatrix}
  \frac{1}{\tau}  I - \tau A^\top A &  0   \\ 
    0 &  \frac{1}{\tau} I - \tau AA^\top 
\end{bmatrix} .
\]
 Then, $\cS,\cG \in \bbS_{++}$ requires $\tau \in\ ]0,  \frac{1}{ \|A\| }[$. The weak convergence follows from Remark \ref{r_gppa}-(iii).
\hfill
\end{proof}

\vskip.1cm
The condition of original \cite[Proposition 4.2]{arias_2011} is $\tau\in [\epsilon, \frac{1-\epsilon}{\|A\|} ]$ with $\epsilon\in\ ]0, \frac{1}{\|A\|+1}[$, which is equivalent to our result, when $\epsilon \rightarrow 0^+$.

\begin{example} [Generalized Dykstra-like algorithm \cite{plc_dual_2010}]
\label{eg_plc_dual}
Consider \cite[Problem 1.2]{plc_dual_2010}
\[
\min_u h(u) + g(A u -r)+
\frac{1}{2} \|u - w\|^2,
\]
which is a special case of \eqref{p} with $m=1$, $l=\iota_{ \{  0\}  } $,  $f = \frac{1}{2} \| \cdot - w\|^2$,  $ z = 0$.  Assuming $r\in\sri (A(\dom h) - \dom g)$,   \cite[Algorithm 3.5]{plc_dual_2010} is given as 
\be \label{x45}
\left\lfloor \begin{array}{lll}
u^{k}  &: = & \prox_{h} \big( w - A^\top s^k  \big), \\
 s^{k+1} & := & s^k +\gamma \big( 
\prox_{\tau g^*} ( s^k + \tau ( Au^k - r) ) - s^k \big). 
\end{array} \right.
\ee
By \cite[Proposition 3.1]{plc_dual_2010}, the dual problem is $\min_s q(s) +g^*(s) +\langle s|r \rangle$, where $q(s)= \frac{1}{2} \big\| w - A^\top s \big\|^2 - \big( \inf_u h(u) +\frac{1}{2} \big\|u - w +  A^\top  s \big\|^2 \big)$.
By \cite[Theorem 3.7]{plc_dual_2010},  \eqref{x45} is equivalent to a simple proximal FBS for solving the dual problem:
\be \label{w4}
s^{k+1} :=  s^k+ \gamma \big( \prox_{\tau \gtilde^*}
\big(  s^k - \tau \nabla q( s^k) \big) -  s^k \big),
\ee
where $\gtilde^* = g^* +\langle \cdot |r\rangle$. This algorithm fits the relaxed G-FBS operator \eqref{rgfbs} with
\[
x = s, \  \cA= \partial \gtilde^* , \  \cB= \nabla q , \  
\cQ = \frac{1}{\tau} \cI. 
\]
\end{example}

\begin{proposition}
Let $\{ (s^k, u^{k})\}_{k\in\N}$ be a sequence generated by \eqref{x45}. {\red Under Assumption \ref{assume_4}},  if $\tau \in\ ]0, 2/ \|A\|^2[$,  $\gamma \in\ ]0,2-\frac{\tau}{2} \|A\|^2[ $,  the following hold.
\begin{itemize}
\item[\rm (i)] $ s^k \weak s^\star$,  as $k \rightarrow \infty$;

\item[\rm (ii)] $\big\| s^{k+1 } - s^{k }  \big\| 
\le \frac{1}{ \sqrt{k+1 } } \sqrt{\frac{2\gamma}{4-2\gamma -\tau \|A\|^2} } 
\big\|s^{0} - s^\star \big\|$,
$\forall k \in \N$. 
\end{itemize}
\end{proposition}
\begin{proof}
$\cB = \nabla q$ is $\|A\|^2$-Lipschitz continuous, and thus, $\frac{1}{\|A\|^2}$-cocoercive (i.e., $\beta = \|A\|^2$). Then, the results follow by Corollary \ref{c_rgfbs}.
\hfill
\end{proof}

\vskip.1cm
Our convergence condition is milder than the original version of 
$\tau \in \ ]0, 2/\|A\|^2[$ and $\gamma \in \ ]0, 1]$ presented in \cite[Theorems 3.6 and 3.7]{plc_dual_2010} and \cite[Theorem 3.4]{plc}.  Moreover, since the variable $u$ is merely intermediate update of \eqref{x45}, and is afterwards removed in \eqref{w4}. The convergence of $\{u^k\}_{k\in\N}$ cannot be concluded by Corollary \ref{c_rgfbs}, which requires additional work, see \cite[Theorem 3.7-(ii)]{plc_dual_2010}.

\begin{example} [PAPC \cite{zxq_ip,teboulle_2015,gist}]
Consider $\min_u f(u) + g(Au)$, which is
 a special case of \eqref{p} with $m=1$,  $h=  0$, $l=\iota_{ \{  0 \} }$, $r = 0 $, $z = 0$. Assuming $f+g$ is coercive, the PAPC scheme is given as
\be \label{papc}
\left\lfloor \begin{array}{lll}
s^{k+1}  & := & \prox_{\sigma g^*} \big( 
(I - \sigma \tau A A^\top ) s^k +\sigma A
(u^k - \tau \nabla f(u^k))  \big), \\
u^{k+1} & := & u^k - \tau \nabla f(u^k) - \tau A^\top
s^{k+1},
\end{array} \right.
\ee
which can be interpreted by  the G-FBS operator \eqref{gfbs}:
\[
x^k =    \begin{bmatrix} 
s^{k }  \\ u^k   \end{bmatrix}, \  
\cA =   \begin{bmatrix}
\partial g^*   &  -A  \\
  A^\top &   0   \end{bmatrix},\ 
\cB =  \begin{bmatrix}
  0  &   0     \\   0   & \nabla f \end{bmatrix},\  
 \cQ =   \begin{bmatrix}
\frac{1}{\sigma} I - \tau A A^\top &  0   \\
  0   & \frac{1}{\tau}  I  \end{bmatrix}.
\] 
\end{example}

\begin{proposition}
Let $\{ (s^k, u^{k})\}_{k\in\N}$ be a sequence generated by \eqref{papc}. {\red Under Assumption \ref{assume_4}}, if $0< \tau < \min \big\{\frac{2}{\beta}, \frac{2-\beta\sigma} {2\sigma \|A^\top A\|} \big\}$,  
 the following hold.
\begin{itemize}
\item[\rm (i)] $ (s^k, u^{k}) \weak (s^\star,u^\star) \in \zer (\cA+\cB)$.

\item[\rm (ii)] 
$\big\| \bb^{k+1 } - \bb^{k }  \big\|_\cQ 
\le \sqrt{ \frac{2\nu}{ (k+1) (2\nu -\beta)} }
\big\|\bb^{0} -\bb^\star \big\|_\cQ$,
where $\nu = \min\big\{ \frac{1}{\sigma} - \tau \|A^\top A\|, 
\frac{1}{\tau} \big\}$. 
\end{itemize}
\end{proposition}
\begin{proof}
Since $f+g$ is coercive, then $\Arg\min (f+g) \ne \emptyset$ by \cite[Proposition 3.1-(i)]{plc}. By Theorem \ref{t_dist}, 
the convergence condition follows from $\cQ \succ \frac{\beta}{2}\cI $.  The value of $\nu$ chosen in (ii) satisfies $\cQ \succeq \nu\cI$.
\hfill 
\end{proof}

\begin{example} [AFBA   \cite{latafat_2017,latafat_chapter}] \label{afba}
Consider $\min_u f(u) + h(u) + g(Au)$, which is 
 a special case of \eqref{p} with $m=1$,  $l=\iota_{ \{ 0 \} }$, $r =  0$, $z = 0$.  The AFBA scheme is given by
\be \label{afba}
\left\lfloor \begin{array}{lll}
s^{k+1}  & := & \prox_{\sigma g^*} \big( s^k 
+ \sigma Aw^k \big), \\
u^{k+1} & := & w^k - \tau A^\top ( s^{k+1} - s^k),  \\
w^{k+1} & := & \prox_{\tau h} \big( u^{k+1}   - \tau \nabla f (u^{k+1}) -\tau A^\top s^{k+1} \big) . 
\end{array} \right.
\ee
To interpret it by the G-FBS operator,  we   remove $w$ and obtain the equivalent form:
\[
\left\lfloor \begin{array}{lll}
s^{k+1}  & := & \prox_{\sigma g^*} \big( s^k 
+ \sigma A \big(u^{k+1} +\tau A^\top (s^{k+1} - s^k) \big) \big), \\
u^{k+1} & := & \prox_{\tau h} \big( u^{k }   - \tau \nabla f (u^{k }) -\tau A^\top s^{k } \big) -\tau A^\top (s^{k+1} - s^k) ,
\end{array} \right.
\]
which corresponds exactly to the relaxed G-FBS operator \eqref{t_relaxed} with 
\[
\bb^k = \begin{bmatrix}
\bs ^{k } \\  u^{k }  \end{bmatrix},\    
\cA =   \begin{bmatrix}
\partial g^*   & -A   \\
A^\top & \partial h \end{bmatrix}, \  
\cB =  \begin{bmatrix}
 0  &   0     \\   0   & \nabla f   \end{bmatrix},
\quad  \cQ =   \begin{bmatrix}
\frac{1}{\sigma} I    &  0   \\
 -A^\top     & \frac{1}{  \tau}  I
\end{bmatrix},\ 
\cM =  \begin{bmatrix}
I   &  0   \\    -\tau A^\top     &  I
\end{bmatrix}.
\] 
\end{example}

\begin{proposition}
Let $\{ (s^k, u^{k})\}_{k\in\N}$ be a sequence generated by \eqref{afba}.  {\red Under Assumption \ref{assume_4}}, if $0< \tau < \min \big\{\frac{2}{\beta}, \frac{2-\beta\sigma} {2\sigma \|A^\top A\|} \big\}$,  
 $ (s^k, u^{k}) \weak (s^\star,u^\star) \in  \zer (\cA+\cB)$. 
\end{proposition}
\begin{proof}
By Remark \ref{r_gppa}-(iii), $\cS$ and $\cG$ are computed as
\[
\cS =  \begin{bmatrix}
   \frac{1}{ \sigma} I    & 0  \\
  0  & \frac{1}{\tau } I \end{bmatrix},\  
  \cG = \begin{bmatrix}
  \frac{1}{\sigma}  I - \tau A A^\top -\frac{\beta}{2} I &  0   \\ 
    0 &  \frac{1}{\tau} I - \frac{\beta}{2} I 
\end{bmatrix} .
\]
Then, by Theorem \ref{t_gppa}, the convergence condition follows from $\cS,\cG \in \bbS_{++}$. 
\hfill 
\end{proof}

\vskip.1cm
The convergence condition of AFBA \eqref{afba} is same as that of PAPC \eqref{papc}. This result is consistent with \cite[Eq.(5.10)]{latafat_chapter} and \cite[Proposition 5.2]{latafat_2017}.

\begin{example} [Condat algorithm \cite{condat_2013}]
Consider the problem $\min_u f(u) + h(u) +\sum_{i=1}^m g_i(A_i u)$, which is a special case of \eqref{p} with $l_i = \iota_{\{   0 \} }$, $r =  0$, $z = 0$. 
\cite[Algorithms 3.1 and 5.1]{condat_2013} is given by
\be \label{condat}
\left\lfloor \begin{array}{lll}
\utilde^{k}  & = & \prox_{\tau h} \big(u^k - \tau
\nabla f(u^k) - \tau \sum_{i=1}^m  A_i^\top s_i^k \big), \\
\stilde_i^{k}  & = & \prox_{\sigma g_i^*} \big( s_i^k +\sigma   A_i ( 2\utilde^k - u^k ) \big), \quad  (i=1,2,...,m)  \\
u^{k+1}  & = & u^k +\gamma  (\utilde^k - u^k), \\
s_i^{k+1}  & = & s_i^k +\gamma  (\stilde_i^k - s_i^k), \quad (i=1,2,...,m)  
\end{array} \right.
\ee
which can be compactly expressed in a relaxed G-FBS form \eqref{rgfbs}:
\[
x^{k} = \begin{bmatrix}
u^k \\ s_1^k \\ \vdots \\ s_m^k
\end{bmatrix},\ 
\cA =   \begin{bmatrix}
\partial h   & A_1^\top & \cdots & A_m^\top \\
- A_1 & \partial g_1^* & \cdots &  0 \\
\vdots  & \vdots  & \ddots & \vdots  \\
- A_m &  0  & \cdots & \partial g_m^* \\
  \end{bmatrix},\ 
\cB =  \begin{bmatrix} 
\nabla f &  0 & \cdots &  0  \\ 
 0 &  0 & \cdots & 0 \\
\vdots  & \vdots  & \ddots & \vdots  \\
 0 &  0 & \cdots &  0 
  \end{bmatrix},
\]
\[   
\cQ =    \begin{bmatrix}
\frac{1}{\tau} I   &  -A_1^\top & \cdots & - A_m^\top \\
- A_1 & \frac{1}{ \sigma} I  & \cdots &  0 \\
\vdots  & \vdots  & \ddots & \vdots  \\
- A_m &  0  & \cdots & \frac{1}{ \sigma} I \\
  \end{bmatrix}.
\] 
\end{example}

\begin{proposition} \label{p_condat}
Let $\{ (u^k, s_1^k, ..., s_m^k)\}_{k\in\N}$ be a sequence generated by \eqref{condat}. {\red Under Assumption \ref{assume_4}},  if  $\tau \in\ ]0, \frac{1}{\sigma \|\sum_{i=1}^m A_i^\top A_i\|} [$, $\gamma \in\ ]0, 2-\frac{\beta}{2} (\frac{1}{\tau} -\sigma \|\sum_i^m A_i^\top A_i\|)^{-1} [$,   the following hold.
\begin{itemize}
\item[\rm (i)]  $ (u^{k},s_1^k,...,s_m^k) \weak (u^\star,s_1^\star,...,s_m^\star) \in \zer( \cA+\cB)$ as $k \rightarrow \infty$.

\item[\rm (ii)]  
$\big\| \bb^{k+1 } - \bb^{k }  \big\|_\cQ 
\le \frac{1}{\sqrt{k+1}}  \sqrt{ \frac{2\gamma\nu}{ (4-2\gamma)\nu -\beta}} 
\big\|\bb^{0} -\bb^\star \big\|_\cQ$,
where $\nu = \frac{1}{\tau} - \sigma \big\| \sum_{i=1}^m A_i^\top A_i \big\|$.
\end{itemize}
\end{proposition}
\begin{proof}
The condition $\cQ \in \bbS_{++}$ requires $\tau\sigma \in\ ]0, \frac{1}{\|\sum_i A_i^\top A_i \|} [ $, and $\nu$ could be chosen as $ \frac{1}{\tau} - \sigma \big\| \sum_{i=1}^m A_i^\top A_i \big\|$, such that $\cQ \succeq \nu \cI$. Then the results follow by Corollary \ref{c_rgfbs}.
\hfill 
\end{proof}

\vskip.1cm
This result is consistent with  \cite[Theorem 5.1]{condat_2013}.  Another counterpart  algorithm to \eqref{condat} is given by \cite[Algorithms 3.2 and 5.2]{condat_2013}:
\be \label{t6}
\left\lfloor \begin{array}{lll}
\stilde_i^{k}  & = & \prox_{\sigma g_i^*} \big(s_i^k + \sigma   A_i  u^k  \big),  \quad (i=1,2,...,m)  \\
\utilde^{k}  & = & \prox_{\tau h} \big(u^k - \tau
\nabla f(u^k) - \tau \sum_{i=1}^m 
 A_i^\top (2\stilde_i^k - s_i^k) \big) ,\\
u^{k+1}  & = & u^k +\gamma  (\utilde^k - u^k) ,\\
s_i^{k+1}  & = & s_i^k +\gamma  (\stilde_i^k - s_i^k). \quad (i=1,2,...,m)  
\end{array} \right.
\ee
The corresponding relaxed G-FBS form is with the same $\bb$, $\cA$, $\cB$ and $\cM$ as the above example, and $\cQ$ is given as 
\[   
\cQ =    \begin{bmatrix}
\frac{1}{ \tau }  I    & A_1^\top & \cdots & A_m^\top \\
 A_1 & \frac{1}{ \sigma } I  & \cdots & 0 \\
\vdots  & \vdots  & \ddots & \vdots  \\
A_m &  0  & \cdots & \frac{1}{ \sigma }  I\\
  \end{bmatrix}.
\]
Proposition \ref{p_condat} also applies for \eqref{t6}.

\subsection{Other examples}
Other classes of algorithms can also be expressed by the G-FBS operator \eqref{gfbs}.   Let us now consider a typical  optimization problem with a linear equality constraint:
\be \label{problem_alm}
\min_u h(u),\qquad \text{s.t.\ } Au = c , 
\ee
where $A: \cX \mapsto \cY$, the function  $h: \cX\mapsto \R\cup \{+\infty\}$ is proper, l.s.c. and convex.

\begin{example} [Basic ALM] \label{eg_alm}
The augmented Lagrangian method (ALM)  is (see \cite[Eq.(1.2)]{mafeng_2018}, \cite[Eq.(7.2)]{taomin_2018} and \cite[Algorithm 6.1]{hbs_2014} for example)
\be \label{alm}
\left\lfloor \begin{array}{lll}
u^{k+1}   & :\in &  \arg \min_u  h(u) +  \frac{\tau}{2}
\big\| Au - c -  \frac{1}{\tau} s^k  \big\|^2 ,  \\
s^{k+1}  & :=  & s^k - \tau  (Au^{k+1} - c ). 
\end{array} \right. 
\ee
It can be customized by the G-FBS operator with
\[
\bb^k =  \begin{bmatrix}
   u^{k } \\ s^{k } \end{bmatrix},\ 
\cA = \begin{bmatrix}
 \partial h  &  - A^\top  \\  A & \partial l
    \end{bmatrix},\ 
\cB=  0, \ 
 \cQ =  \begin{bmatrix}
 0  &   0   \\
 0  & \frac{1}{\tau} I_M  \end{bmatrix}, 
\] 
where $l = -\langle \cdot |c \rangle$.
\end{example}

\begin{proposition} \label{p_alm}
Let $\{ (u^k, s^k)\}_{k\in\N}$ be a sequence generated by \eqref{alm}. If $\tau >0$,  the following hold.
\begin{itemize}
\item[\rm (i)] $\big\| s^{k+1 } - s^{k }  \big\|
\le \frac{1}{ \sqrt{k+1 } } 
\big\| s^{0} - s^\star \big\|$, $\forall k \in \N$; 

\item[\rm (ii)] $ s^{k} \weak  s^\star$, as $k \rightarrow \infty$;
\item[\rm (iii)] $ Au^{k} \rightarrow c$, as $k \rightarrow \infty$.

\end{itemize}
\end{proposition}
\begin{proof}
{\rm (i)--(ii)} follow from Corollary \ref{c_ppa}. 

{\rm (iii)} follows from the $s$-update of \eqref{alm} and strong convergence of $s^{k+1}-s^k\rightarrow 0$.
\hfill 
\end{proof}

\vskip.1cm
The variable $u$ in \eqref{alm} is merely intermediate result that is not proximally regularized. That is why the metric $\cQ$ is degenerate here. Moreover, the $u$-step of \eqref{alm} may not have the unique solution. However, $Au^k$ is always unique \cite{self_eq}.  This is also observed in \cite[Remark 6.1]{hbs_2014}. The tractability of  the $u$-step can be solved by  linearization, as shown in the next example.

\begin{example} [Linearized ALM] \label{eg_linear_alm}
The linearlized ALM is given as \cite{yang_yuan_2013}
\be \label{linear_alm}
\left\lfloor \begin{array}{lll}
u^{k+1}   & := &  \arg \min_u   h(u) + \frac{\rho } {2} 
\big\| u - u^k +  \frac{1}{ \rho} A^\top \big( \tau  
(A u^k -c) -  s^k \big) \big\|^2,  \\
 s^{k+1}  & :=  & s^k - \tau  (  Au^{k+1} - c) .
\end{array} \right. 
\ee
Its G-FBS  interpretation \eqref{gfbs} is given as
\[
x^k =  \begin{bmatrix}
  u^{k } \\ s^{k } \end{bmatrix},\ 
\cA = \begin{bmatrix}
 \partial h  &  - A^\top  \\  A & \partial l
    \end{bmatrix},\ 
\cB =  0 , \ 
 \cQ =  \begin{bmatrix}
\rho I  -  \tau A^\top A  &  0   \\
 0  & \frac{1}{\tau} I  \end{bmatrix}, 
\] 
where $l=-\langle \cdot| c\rangle$. 
\end{example}

\begin{proposition} \label{p_lalm}
Let $\{ (u^k, s^k)\}_{k\in\N}$ be a sequence generated by \eqref{linear_alm}. If $\tau \in \ ]0, \rho/   \|A^\top A\|[$,  the following hold.
\begin{itemize}
\item[\rm (i)] $\big\| x^{k+1 } - x^{k }  \big\|_\cQ
\le \frac{1}{ \sqrt{k+1 } } 
\big\| x^{0} - x^\star  \big\|_\cQ $,
$\forall k \in \N$.  

\item[\rm (ii)] $(u^k, s^{k}) \weak (u^\star, s^\star) \in \zer \cA$, as $k \rightarrow \infty$, where $u^\star$ is a solution to the problem \eqref{problem_alm}. 

\item[\rm (iii)] $ Au^{k} \rightarrow c$, as $k \rightarrow \infty$.
\end{itemize}
\end{proposition}
\begin{proof}
{\rm (i)--(ii)} follow from Corollary \ref{c_ppa}. 

{\rm (iii)} follows from the $s$-update of \eqref{linear_alm} and strong convergence of $s^{k+1}-s^k\rightarrow 0$.
\hfill 
\end{proof}

\vskip.1cm
The linearization strategy can be viewed as a preconditioning, which guarantees the uniqueness of the $u$-step and the weak convergence of $u^k \weak u^\star$.

\begin{example} [Linearized Bregman algorithm  \cite{cjf_3}]
\label{eg_lb}
The scheme reads as \cite[Eq.(2.2)]{cjf_3}
\be \label{lb}
\left\lfloor \begin{array}{lll}
u^{k+1}   & := &  \arg \min_u \rho \tau  h(u) + 
\frac{1}{2} \big\| u - \rho A^\top  s^k  \big\|^2,   \\
s^{k+1} & :=  & s^k  - (Au^{k+1} - c). 
\end{array} \right. 
\ee
Its G-FBS  interpretation \eqref{gfbs} is given as
\[ 
x^k =    \begin{bmatrix}
  u^{k}  \\ s^{k}  \end{bmatrix},\  
\cA = \begin{bmatrix}
 \tau  \partial \tilde{h} + \frac{1}{\rho} I - A^\top A
   &  -A^\top      \\  A &  \partial l
    \end{bmatrix},\ 
 \cB =  0,\  \cQ =  \begin{bmatrix}
 0  &   0   \\   0  & I  \end{bmatrix} ,
\]
where $\tilde{h} = h+\langle \cdot| A^\top c\rangle$, $l=-\langle \cdot |c \rangle$. 
\end{example}

\begin{proposition}
Let $\{ (u^k, s^k)\}_{k\in\N}$ be a sequence generated by \eqref{lb}. If $\tau>0$ and $\rho \in\ ]0, \frac{1}{\|A^\top A\|} [$,  the following hold.
\begin{itemize}
\item[\rm (i)] $\big\| s^{k+1 } - s^{k }  \big\|
\le \frac{1}{ \sqrt{k+1 } } 
\big\| s^{0} - s^\star \big\|$, $ \forall k \in \N$;

\item[\rm (ii)] $ s^{k}  \weak s^\star$, as $k\rightarrow\infty$;
\item[\rm (iii)] $Au^k  \rightarrow  c$, as $k\rightarrow\infty$.
\end{itemize}
\end{proposition}
\begin{proof}
Similar to Proposition \ref{p_alm} and \ref{p_lalm}.
\hfill
\end{proof}

\subsection{Short summary}
Many first-order operator splitting algorithms have been shown as the customized applications of a simple G-FBS operator \eqref{gfbs} or its relaxed version \eqref{t_relaxed}. One can verify that more existing algorithms, e.g.,  \cite{fang_2015,mafeng_2018,hbs_2014,hbs_yxm_2018,
cch_2016,bai_2018}, belong to the G-FBS class, which are not detailed here.

We  stress the simplicity of the G-FBS fitting, compared to the complicated characterization by variational inequality \cite{hbs_siam_2012,hbs_siam_2012_2} or procedure of constructing Fej\'{e}r monotone sequence \cite{plc_vu}.  More importantly, there is no need to perform the convergence analysis case-by-case: All the results are the   immediate consequences of the general results in Sect. \ref{sec_gfbs} and \ref{sec_extension}.

In many examples listed above, the cocoercive operator $\cB=0$, and the problem \eqref{inclusion} becomes $0\in\cA x$. Notice that it cannot be simply understood as a minimization of only one convex function. In fact, $0\in \cA x$ in our setting also encompasses the minimization of the sum of multiple convex (not necessarily smooth) functions.

Moreover, the monotone operators  $\cA$  in many examples  bear the typical (diagonal) monotone  + (off-diagonal) skew structure: 
\[
\cA =  \begin{bmatrix}
\partial f &    -A^\top  \\
A   & \partial g  \end{bmatrix} = \underbrace{ 
 \begin{bmatrix}
\partial f &  0   \\
  0   & \partial g  \end{bmatrix} }_\text{monotone} +  
\underbrace{ \begin{bmatrix}
  0  & - A^\top  \\  A &   0 
\end{bmatrix} }_\text{skew}, 
\]
which coincides with the observations  in \cite{plc_fixed,arias_2011,bredies_2017}.  It can be further verified that $\cA$ with such structure fails to be cyclically monotone, and cannot be viewed as a subdifferential of a convex function, by \cite[Theorem 22.18]{plc_book}. This was also  mentioned in Sect. \ref{sec_summary} and  Remark \ref{r_assume_2}-(iv). Therefore, there are no conclusions regarding the convergence of objective value in Examples \ref{eg_radmm}--\ref{eg_padmm} and \ref{eg_cp}--\ref{eg_lb}. 

\section{Conclusions}
In this paper, we performed a systematic study of the G-FBS operator and its associated fixed-point iterations, particularly under degenerate setting.  A great variety of operator splitting algorithms  were illustrated as the  concrete examples of the G-FBS operator.

Last, it seems interesting to further extend the proposed framework to the case when the  metric $\cQ$ and relaxation operator $\cM$ are allowed to vary over the iterations. Another  limitation of this G-FBS operator is that $\cQ$ and $\cM$ are assumed as linear here. Thus,  it fails to cover Bregman proximal algorithms  \cite{teboulle_2018,plc_bregman} and a few PDS algorithms  \cite{plc_2012,bot_jmiv_2014,ywt_2017}, which may correspond to nonlinear metric $\cQ$ and $\cM$. It is worthwhile to extend the G-FBS operator \eqref{gfbs} to nonlinear case.

\section{Acknowledgements}
I am gratefully indebted to the anonymous reviewers  for helpful discussions, particularly related to singularity of metric,  closer connections to the existing works, and the simplified proof of Proposition \ref{p_gpfbs_obj}.

\section{Data availability}
There is no associated data with this manuscript.

\section{Disclosure statement}
The author declares there are no conflicts of interest regarding the publication of this paper.

%\subsubsection*{References}
%\pdfbookmark[0]{References}{references} % add a nice PDF bookmark

%\bibliographystyle{siam}
%\bibliographystyle{plain}
\bibliographystyle{unsrt} % order of appearance
\small{
\bibliography{refs}
}

\end{document}